\documentclass[10pt]{amsart}
\usepackage{amssymb,amsmath,amsfonts,amscd}
\usepackage{times,graphicx,epsfig}

\theoremstyle{plain}
\newtheorem{theorem}{Theorem}[section]
\newtheorem{proposition}[theorem]{Proposition}
\newtheorem{definition}[theorem]{Definition}
\newtheorem{notation}[theorem]{Notation}
\newtheorem{corollary}[theorem]{Corollary}
\newtheorem{remark}[theorem]{Remark}
\newtheorem{lemma}[theorem]{Lemma}

\def\eps{\epsilon}

\usepackage{pdfsync}
\newcommand{\E}{\mathbb E}

\newcommand{\rn}[1]{{\mathbb R}^{#1}}
\newcommand{\R}{\mathbb R}

\newcommand{\N}{\mathbb N}

\newcommand{\supp}{\mathrm{supp}\;}

\newcommand{\cov}[1]{{\bigwedge\nolimits^{#1}{\mfrak h}}}

\newcommand{\covw}[2]{{\bigwedge\nolimits^{#1,#2}{\mfrak h}}}

\newcommand{\covH}[1]{{\bigwedge\nolimits^{#1}{\mfrak h}}}
\newcommand{\vetH}[1]{{\bigwedge\nolimits_{#1}{\mfrak h}}}

\newcommand{\covh}[1]{{\bigwedge\nolimits^{#1}{\mfrak h_1}}}

\newcommand{\he}[1]{{\mathbb H}^{#1}}

\newcommand{\scal}[2]{\langle {#1} , {#2}\rangle}

\newcommand{\Scal}[2]{\langle {#1} \vert {#2}\rangle}
\newcommand{\scalp}[3]{\langle {#1} , {#2}\rangle_{#3}}

\newcommand{\ccheck}{{\vphantom i}^{\mathrm v}\!\,}
\newcommand{\mc}{\mathcal }

\newcommand{\mfrak}{\mathfrak}

\newcommand{\WO}[3]{\mathop{W}\limits^\circ{}\!^{{#1},{#2}}
(#3)}

\usepackage[usenames]{color}

\begin{document}

%\today

\bigskip

\title{ {Cohomology of annuli,} duality and $L^\infty$-differential forms on Heisenberg groups}

\author[Annalisa Baldi, Bruno Franchi, Pierre Pansu]{
Annalisa Baldi\\
Bruno Franchi\\ Pierre Pansu
}

\begin{abstract}
In the last few years the authors  proved  Poincar\'e and Sobolev type inequalities in Heisenberg groups  $\he n$ for differential forms  in the Rumin's complex. The need to substitute the usual de Rham complex of differential forms for Euclidean spaces with the Rumin's complex  is due to the different stratification of the Lie algebra of Heisenberg groups. 
The crucial feature of Rumin's complex is that $d_c$ is a differential operator of order 1 or 2 according to the degree of the form.

Roughly speaking, Poincar\'e and Sobolev type inequalities are 
quantitative 
formulations of the well known topological problem whether a closed form is exact.
More precisely, for suitable $p$ and $q$, we mean that  every exact differential form $\omega$ in $L^p$ admits a primitive $\phi$ in $L^q$ such that
$
\|\phi\|_{L^{q}}\leq C\ \|\omega\|_{L^{p}}$.
The cases of the norm $L^p$, $p\ge 1$ and $q<\infty$ have been already studied in a series of papers by the authors. In the present paper we  deal with the limiting case 
where $q=\infty$: it is remarkable that, unlike in the scalar case, when the degree of the forms $\omega$ is at least $2$, we can take $q=\infty$ in the left-hand side of the inequality. The corresponding inequality in  the Euclidean setting $\mathbb{R}^N$ ($p=N$ and $q=\infty$) was proven by Bourgain \& Brezis.
\end{abstract}
 


\maketitle

\tableofcontents

\section{Introduction}\label{introduction}

\subsection{Euclidean spaces and de Rham complex}\label{derham introduction}

To begin with, let us consider preliminarily the Euclidean space $\mathbb R^N$, $N>1$ and the differential forms
of the de Rham complex $(\Omega^\bullet, d)$ on $\mathbb R^N$. 
It is well known that closed forms $\omega
\in \Omega^\bullet$ are exact, i.e. $d\omega=0$ implies that there exists $\phi\in \Omega^{\bullet - 1}$ such that
$d\phi =\omega$. This can expressed by saying that the cohomology groups
\begin{equation}\label{cohomology 1}
H^h_{\mathrm{dR}} : = (\Omega^h\cap \ker d) /d\Omega^{h-1} \qquad\mbox{are trivial for $1\le h\le N$.}
\end{equation}

Global Poincar\'e and Sobolev inequalities in $(\Omega^\bullet, d)$ are meant to give a quantitative meaning
to \eqref{cohomology 1}. More precisely, if $1\le p, q\le\infty$,
we say that a \emph{global} Poincar\'e inequality holds on $\mathbb R^N$, if 
there exists a positive constant $C=C(p,q)$ such that for every exact $h$-form $\omega$ on $\mathbb R^N$, belonging to $L^p$, 
there exists a $(h-1)$-form $\phi$ such that $d\phi=\omega$ and
\begin{eqnarray*}
\|\phi\|_{L^q} \leq C\,\|\omega\|_{L^p}.
\end{eqnarray*} 
Shortly, we shall write that Poincar\'e$_{p,q}(h)$ inequality holds, or that the $L^{q,p}$-cohomology
vanishes. For further comments and applications, we refer to \cite{Pcup}.
Notice that a homogeneity argument shows that, if $1\le p< N$, then we can take
 $p\le q\le pN/(N-p)$.

In addition, we say that a \emph{global} Sobolev inequality holds on $\mathbb R^N$, if for every exact compactly supported 
$h$-form $\omega$ on $\mathbb R^N$, belonging to $L^p$, there exists a compactly supported $(h-1)$-form $\phi$ such that $d\phi=\omega$ and
\begin{eqnarray*}
\|\phi\|_{L^q} \leq C\,\|\omega\|_{L^p}.
\end{eqnarray*}
Again, we shall write that Sobolev$_{p,q}(h)$ holds.

We point out that if $u$ is a scalar function on $\mathbb R^N$ (i.e. $u\in\Omega^0$), then Poincar\'e$_{p,q}(1)$ and
Sobolev $_{p,q}(1)$ for $du$ are nothing but the usual Poincar\'e and Sobolev inequalities.

Besides global inequalities, it is natural to consider {\sl local} inequalities, where the
Euclidean space $\mathbb R^N$ is replaced (for instance) with a Euclidean ball.
If $1<p<N$, a local Poincar\'e inequality in de Rham complex has been proved by
Iwaniec \& Lutoborsky \cite{IL}, and a Sobolev inequality for bounded convex sets
has been proved by Mitrea, Mitrea \& Monniaux \cite{mitrea_mitrea_monniaux}.

The notions of Poincar\'e and Sobolev inequalities can be weakened through the notions of
\emph{interior  inequalities}. More precisely, we say
that  an interior 
Poincar\'e$_{p,q}(h)$ inequality holds on $\mathbb R^N$ if there exists a fixed
$\lambda\geq 1$ large enough such that
for every $r>0$  small enough  
there exists a constant $C=C(M,p,q,r,\lambda)$ such that for every $x\in \mathbb R^N$ and every exact $h$-form $\omega$ on 
$B(x,\lambda r)$, belonging to $L^p$, there exists a $(h-1)$-form $\phi$ on $B(x,r)$ such that $d\phi=\omega$ on $B(x,r)$ and
\begin{eqnarray}\label{int poinc}
\|\phi\|_{L^q(B(x,r))} \leq C\,\|\omega\|_{L^p(B(x,\lambda r))}.
\end{eqnarray} 
Analogously, by \emph{interior Sobolev inequalities}, we mean that, if $\omega$ is supported in $B(x,r)$, then there exists $\phi$ supported in $B(x,\lambda r)$ such that $d\phi = \omega$ and 
\begin{eqnarray}\label{int sob}
\|\phi\|_{L^q(B(x,\lambda r))} \leq C\,\|\omega\|_{L^p(B(x, r))}.
\end{eqnarray} 
{ Here we use the word  \emph{interior} to stress the fact that inequality \eqref{int poinc} provides no information on the behaviour of differential forms near the boundary of their domain of definition.}

It turns out that in several situations, the loss on domain is harmless. This is for instance the case of $L^{q,p}$-cohomological applications, see \cite{Pcup}.
 
Relying on these weaker notions, we have been able to cover also the case $p=1$ (see \cite{BFP4}). The other endpoint result $p=N$, $q=\infty$, is more delicate. Indeed, it is well known that the interior
$\mathrm{Poincar\acute{e}}_{N, \infty}(1)$ fails to hold in  $\mathbb R^N$ (see e.g. \cite{wheeden},  p. 484),
and has to be replaced by the so-called Trudinger exponential estimate (see \cite{trudinger})
or by the more precise Adams-Trudinger inequality (\cite{wheeden}, Theorem 15.30).

However, rather surprisingly, in  \cite{BB2007} Bourgain \& Brezis proved that a global $\mathrm{Poincar\acute{e}}_{N, \infty}(h)$
holds for $1<h<N-1$. 

\subsection{Heisenberg groups and Rumin's complex}\label{heisenberg intro} 

In the last few years, the authors of the present paper have attacked the study of Poincar\'e and Sobolev inequalities in sub-Riemannian manifolds endowed with a ``suitable'' complex of differential forms (we remind that the data of a smooth manifold $M$ and of a sub-bundle $H\subset TM$ equipped with a scalar product $g$ is called a \emph{sub-Riemannian} manifold).
See, e.g., \cite{GromovCC}, \cite{montgomery}.
 
 More precisely, we have considered
 differential forms of the so-called Rumin complex of Heisenberg groups: see \cite{BF5}, \cite{BFP1},
 \cite{BFP2}, \cite{BFP3}, \cite{BFP_Catania}.  
The Heisenberg group $\he n$, $n\geq 1$, is the connected, simply connected Lie group whose Lie algebra is the central extensions 
\begin{equation}\label{strat intro}
\mathfrak{h}=\mathfrak{h}_1\oplus\mathfrak{h}_2,\quad\mbox{with $\mathfrak{h}_2=\R=Z(\mathfrak{h})$,}
\end{equation}
with bracket $\mathfrak{h}_1\otimes\mathfrak{h}_1\to\mathfrak{h}_2=\R$ being a non-degenerate skew-symmetric 2-form. 
 Due to its 
stratification \eqref{strat intro}, the Heisenberg Lie algebra admits a one parameter group of automorphisms $\delta_t$,
\begin{eqnarray*}
\delta_t=t\textrm{ on }\mathfrak{h}_1,\quad \delta_t=t^2 \textrm{ on }\mathfrak{h}_2,
\end{eqnarray*}
which are counterparts of the usual Euclidean dilations in $\rn N$. Through exponential coordinates, $\he n$ can be identified  
with the Euclidean space $\mathbb R^{2n+1}$, endowed with the
non-commutative product induced by the Campbell-Hausdorff formula. In this system of coordinates,
the identity element $e\in \he n$ is the zero of the vector space $\mathbb R^{2n+1}$, and $p^{-1} = -p$.
In addition, in this system of coordinates, the Haar measure of the group is the $(2n+1)$-dimensional Lebesgue measure $\mc L^{2n+1}$.

Heisenberg groups are the simplest nontrivial (i.e. non-commutative) instance of
the so called {\sl Carnot groups},
connected, simply connected and stratified Lie groups. 
Heisenberg groups can be viewed as sub-Riemannian spaces,
where the sub-Riemannian structure is obtained by left-translating $\mathfrak{h}_1$ (we remind that
the Lie algebra of $\he n$ can be identified with the tangent
space to $\he n$ at $e$).
In addition, Heisenberg groups are the local models of contact manifolds,
since, according to a theorem by Darboux, every $2n+1$-dimensional contact manifold is locally contactomorphic to $\he n$.

For a general review on Heisenberg groups and their properties, we
refer for instance to \cite{Stein}, \cite{GromovCC}, \cite{BLU},  \cite{VarSalCou}. The main properties
of $\he n$ that we shall need in this paper will be presented below in Section \ref{preliminary}.
Here we limit ourselves to reminding that Heisenberg groups carry natural left-invariant metrics, either Carnot-Carath\'eodory distances as sub-Riemannian manifolds
or, equivalently, Cygan-Kor\'anyi norms $\rho$ (see \eqref{gauge} below). Throughout this paper we use systematically
the Cygan-Kor\'anyi distance $d(p,q) := \rho(p^{-1}\cdot q)$. The distance $d$ is homogeneous of degree
one with respect to group dilations $\delta_t$, so that, if we denote by $B(p,r)$ the Cygan-Kor\'anyi
ball of radius $r>0$ centered at $p\in\he n$, then $\mc L^{2n+1}(B(p,r)) = c r^{2n+2}$. In particular,
this implies that the Hausdorff dimension of $\he n$ with respect to $d$ equals $Q:=2n+2$.

As a consequence of the stratification
\eqref{strat intro}, the differential forms on $\mathfrak{h}$ split into 2 ei\-gen\-spac\-es under $\delta_t$, therefore de Rham complex lacks scale invariance under these anisotropic dilations. A substitute for de Rham's complex, that recovers scale invariance under $\delta_t$ has been defined by M. Rumin, \cite{rumin_jdg}. 

Let $h=0,\ldots,2n+1$. Rumin's substitute for smooth differential forms of degree $h$ are the smooth sections of a left-invariant vector bundle $E_0^h$. If $h\leq n$, $E_0^h$ is a subbundle of $\Lambda^h H^*$. If $h\geq n$, $E_0^h$ is a subbundle of $\Lambda^h H\otimes (TM/H)$. Rumin's substitute for de Rham's exterior differential is a linear differential operator $d_c$ from sections of $E_0^h$ to sections of $E_0^{h+1}$
such that $d_c^2=0$.  Further details about Rumin's complex are contained in Section \ref{rumin complex} below. 
We refer to also to \cite{rumin_jdg}, \cite{BFTT}  \cite{BFT2} and \cite{BFT3} for details of the construction.

\medskip

We stress that the operator $d_c$ is a left-invariant differential
operator of order $2$ when acting on forms of degree $n$ and of order $1$ otherwise.

\medskip

This phenomenon will be a major issue in our results and will affect the proofs (think for instance of Leibniz formula) as well as the choice of the exponents $p,q$ in our inequalities.
 
\medskip

\subsection{Poincar\'e and Sobolev inequalities: precise definitions}
We can state now the notions of (global and interior) Poincar\'e and Sobolev inequalities
in the setting of Rumin's complex.

\medskip

\begin{definition}\label{poincare global def}

 If $1\le h\le 2n+1$ and $1\le p\le q \le \infty$, we say that the global $\he{}$-$\mathrm{Poincar\acute{e}}_{p,q}$ inequality holds in $E_0^h$ if 
there exists a constant $C$ such that,
 for every $d_c$-exact differential $h$-form $\omega$ in $L^p(\he n;E_0^h)$ there exists a differential $(h-1)$-form $\phi$ in $L^q(\he n,E_0^{h-1})$ such that $d_c\phi=\omega$ and
\begin{eqnarray}\label{H global poincare}
\|\phi\|_{L^{q}(\he n,E_0^{h-1})}\leq C\,\|\omega\|_{L^{p}(\he n,E_0^h)} \quad \mbox{ 
$\mathrm{global}\, \he{}$-$\mathrm{Poincar\acute{e}}_{p,q}(h)$.
}
\end{eqnarray}
\end{definition}

\begin{definition}\label{poincare def}

{ Let $B:=B(e,1)$ and $B_\lambda:=B(e,\lambda)$.}
Given $1\le h\le 2n+1$ and $1\le p\le q \le \infty$, we say that the interior $\he{}$-$\mathrm{Poincar\acute{e}}_{p,q}$ inequality holds in $E_0^h$ if 
{ there exist constants $\lambda>1$ and $C$} such that,
 for every $d_c$-exact differential $h$-form $\omega$ in $L^p(B_\lambda;E_0^h)$ there exists a differential $(h-1)$-form $\phi$ in $L^q(B,E_0^{h-1})$ such that $d_c\phi=\omega$ and
\begin{eqnarray}\label{H poincare}
\|\phi\|_{L^{q}(B,E_0^{h-1})}\leq C\,\|\omega\|_{L^{p}(B_\lambda,E_0^h)} \quad 
\mbox{ 
$\mathrm{interior}\, \he{}$-$\mathrm{Poincar\acute{e}}_{p,q}(h)$.
}
\end{eqnarray}
\end{definition}

\begin{definition} \label{equiv global Sobolev}
If $1\le h\le 2n+1$, $1\le p \le q \le \infty $,
we say that the global
$\he{}$-$\mathrm{Sobolev}_{p,q}(h)$ inequality holds if there exists a constant $C$ such that 
 for every compactly supported $d_c$-exact differential $h$-form $\omega$ in $L^p(\he n;E_0^h)$ there exists a 
compactly supported differential $(h-1)$-form $\phi$ in $L^q(\he n,E_0^{h-1})$ such that $d_c\phi=\omega$ 
and
\begin{eqnarray}\label{H global Sobolev}
\|\phi\|_{L^{q}(\he n,E_0^{h-1})}\leq C\,\|\omega\|_{L^{p}(\he n,E_0^h)} \quad 
\mbox{ 
$\mathrm{global}\, \he{}$-$\mathrm{Sobolev}_{p,q}(h)$.
}
\end{eqnarray}
\end{definition}

\begin{definition} \label{equiv Sobolev}
{ Let $B:=B(e,1)$ and $B_\lambda:=B(e,\lambda)$.}
Given $1\le h\le 2n$, $1\le p \le q \le \infty $,
we say that the interior
$\he{}$-$\mathrm{Sobolev}_{p,q}(h)$ inequality holds if { there exist constants $\lambda>1$ and $C$} such that 
 for every compactly supported $d_c$-exact differential $h$-form $\omega$ in $L^p(B;E_0^h)$ there exists a 
 compactly supported differential $(h-1)$-form $\phi$ in $L^q(B_\lambda,E_0^{h-1})$ such that $d_c\phi=\omega$ in $B_\lambda$
and
\begin{eqnarray}\label{H Sobolev}
\|\phi\|_{L^{q}(B_\lambda,E_0^{h-1})}\leq C\,\|\omega\|_{L^{p}(B,E_0^h)} \quad 
\mbox{ 
$\mathrm{interior}\, \he{}$-$\mathrm{Sobolev}_{p,q}(h)$.
}
\end{eqnarray}
Here we have extended $\omega$ by $0$ to all of $B_\lambda$.
\end{definition}

\begin{remark}\label{sobolev local global} As in \cite{BFP2}, Corollary 5.21, an elementary
scaling argument shows that, if $h\neq n$, $1\le p <Q$ and $q= pQ/(Q-p)$, or
$h=n$, $1\le p <Q/2$ and $q= pQ/(Q-2p)$ then the { interior}
$\he{}$-$\mathrm{Sobolev}_{p,q}(h)$ implies the global $\he{}$-$\mathrm{Sobolev}_{p,q}(h)$ 
inequality. 

Suppose $1< h <2n$. If $h\neq n$, $p=Q$ take $q=\infty$, and,
$h=n$, $p=Q/2$ take $q=\infty$. We shall see later that (unlike in the case $h=1$ or $h=2n$),
{ interior} $\he{}$-$\mathrm{Sobolev}_{Q,\infty}(h)$ 
hold. Then, again the corresponding global inequalities hold, thanks to the same scaling argument. 

In the sequel, we shall refer to the exponents $q= pQ/(Q-p)$ or $q= pQ/(Q-2p)$ according
to the degree of the forms as to the {\sl sharp Sobolev exponent}.

\end{remark}

\subsection{State of the art}\label{state of the art}

In \cite{BFP2} and \cite{BFP3} the following Poincar\'e and Sobolev inequalities have been
proven. More precisely, \cite{BFP2} deals with the case $p>1$, whereas \cite{BFP3} covers the case $p=1$.

\bigskip

\begin{theorem}[Poincar\'e inequality] \label{P introduction}
If \, $1\le  h \le 2n+1$,
we have:
\begin{itemize}
\item[i)] if $h \neq n+1,2n+1$ and $1\le p< Q$, then the interior
$\he{}$-$\mathrm{Poincar\acute{e}}_{p,pQ/(Q-p)}(h)$ holds;
\item[ii)] if $h = n+1$ and $1\le p< Q/2$, then the interior
$\he{}$-$\mathrm{Poincar\acute{e}}_{p,pQ/(Q-2p)}(n+1)$ holds;
{ \item[iii)] if $h =2n+1$ and $1<p< Q$, then the interior
$\he{}$-$\mathrm{Poincar\acute{e}}_{p,pQ/(Q-p)}(h)$ holds.}
\end{itemize}
Analogous statements hold for global Poincar\'e inequalities on $\he n$.
\end{theorem}

\begin{theorem}[Sobolev inequality]\label{S introduction} If \, $1\le  h \le 2n+1$,
we have:
\begin{itemize}
\item[i)] if $h \neq n+1,2n+1$ and $1\le p< Q$, then the interior
$\he{}$-$\mathrm{Sobolev}_{p,pQ/(Q-p)}(h)$ holds;
\item[ii)] if $h = n+1$ and $1\le p< Q/2$, then the interior
$\he{}$-$\mathrm{Sobolev}_{p,pQ/(Q-2p)}(n+1)$ holds;
{ \item[iii)] if $h =2n+1$ and $1<p< Q$, then the interior
$\he{}$-$\mathrm{Sobolev}_{p,pQ/(Q-p)}(h)$ holds.}
\end{itemize}
Analogous statements hold for global Sobolev inequalities on $\he n$.
\end{theorem}

\bigskip

\subsection{Main results and sketch of the proofs}\label{main results intro}

The aim of the present paper is to complete the results gathered in Section \ref{state of the art},
by covering (when possible) the endpoints $p=Q$ or $p=Q/2$ according to the degree of the forms.

Thus, the core of the present paper consists of the following theorems:

\begin{theorem}\label{poincare infty}
If \, $2\le  h \le { 2n+1}$,
we have:
\begin{itemize}
\item[i)] if $h \neq n+1$, then the interior
$\he{}$-$\mathrm{Poincar\acute{e}}_{Q, \infty}(h)$ holds;
\item[ii)] if $h = n+1$, then the interior
$\he{}$-$\mathrm{Poincar\acute{e}}_{Q/2, \infty}(n)$ holds.
\end{itemize}
Analogous statements hold for global Poincar\'e inequalities on $\he n$.
\end{theorem}

\begin{theorem}
\label{sobolev q infty}
If\, $2\le  h \le { 2n+1}$,
we have:
\begin{itemize}
\item[i)] if $h \neq n+1$, then the interior
$\he{}$-$\mathrm{Sobolev}_{Q, \infty}(h)$ holds;
\item[ii)] if $h = n+1$, then the interior
$\he{}$-$\mathrm{Sobolev}_{Q/2, \infty}(n)$ holds.
\end{itemize}
Analogous statements hold for global Sobolev inequalities on $\he n$.
\end{theorem}

\begin{remark}\label{h=1} In Euclidean space $\mathbb R^N$ it is well known that the interior
$\mathrm{Poincar\acute{e}}_{N, \infty}(1)$ fails to hold (see e.g. \cite{wheeden},  p. 484),
and has to be replaced with Trudinger's exponential estimate (see \cite{trudinger})
or the more precise Adams-Trudinger inequality (\cite{wheeden}, Theorem 15.30).
Analogous estimates in Heisenberg groups can be found e.g. in  \cite{balogh_manfredi_tyson}, \cite{cohn_lu} (see also \cite{BFP_Catania}).

On the contrary, the statement of Theorem \ref{poincare infty} states the $\he{}$-$\mathrm{Poincar\acute{e}}_{p, \infty}(h)$
holds for $h\ge 2$ with sharp exponent $p=Q$ or $p=Q/2$, according to the degree of the forms. 

We refer to \cite{BB2007}, Theorem 5, for related statements in Euclidean
spaces. 

\end{remark}

\begin{remark}\label{euclidean local} By the way, the proofs presented here can be
carried out (with obvious simplifications) also in the Euclidean setting. In particular,
we can obtain { interior} estimates that are, at least partially, the { interior} counterparts of the
results of \cite{BB2007}.
\end{remark}

\medskip 

Let us give a sketch of the paper.  Section \ref{preliminary}
gathers the basic notions about Heisenberg groups. In particular, in Section \ref{sobolev kernels}
we state several properties of the convolution kernels in $\he n$. These  results are more or less
known, but they have to be handled carefully because of the presence of $L^\infty$-spaces,
precluding density arguments. Subsequently, Section \ref{rumin complex} contains a
minimal presentation of Rumin's complex, with the aim of making the paper as
self-contained as possible. In particular, Lemma \ref{leibniz} deals with Leibniz formula
when the exterior differential $d_c$ of the complex is a differential operator of order 2.
A more extensive presentation of the complex is contained in the Appendix (Section \ref{appendix}).
Finally, Section \ref{sect rumin laplacian} contains some basic properties of Rumin's
Laplacian in $(E_0^\bullet,d_c)$ and of its fundamental solution.

The core of the present paper is contained in Sections \ref{caldo}
-- \ref{proof}
 where the proofs of  Theorems \ref{poincare infty} and \ref{sobolev q infty} are carried out.

More precisely, Section \ref{caldo} produces our proof of Poincar\'e inequalities (Theorem \ref{poincare infty}) which relies both on the formulation by duality of the Poincar\'e inequalites $\he{}$-$\mathrm{Poincar\acute{e}}_{p, \infty}(h)$.
Therein we  prove a relationship between  
$\he{}$-$\mathrm{Sobolev}_{ 1,p'}(2n+2-h) $ and $\he{}$-$\mathrm{Poincar\acute e}_{p,\infty}(h) $
when $1<p, p'<\infty$ are dual exponents, and eventually we  combine this relationship
with the previous duality argument for Poincar\'e inequalites.

Contrary to what happens when  $1\le p<Q$ (or $1\le p<Q/2$), where the proofs of Poincar\'e and Sobolev inequalities proceeded on parallel tracks, here the proof of Sobolev inequalities require more delicate arguments: first we prove 
a $L^\infty$-homotopy formula (see Section \ref{section homotopy infty}) and then we derive 
interior $\he{}$-$\mathrm{Poincar\acute{e}}_{\infty, \infty}$ and $\he{}$-$\mathrm{Sobolev}_{\infty, \infty}$
inequalities. Then the subsequent step (Section \ref{homotopy annuli}) consists in proving that the $L^{\infty,\infty}$ cohomology
of Rumin's forms vanishes on a suitable family of (Cygan-Kor\'anyi) annuli. Finally,
Section \ref{proof} contains $L^\infty$-estimates associated with Leibniz formula and hence 
ends with the proof of Sobolev inequalities stated in Theorem \ref{sobolev q infty}.

\section{Heisenberg groups: definitions and preliminary results}\label{preliminary}

 We denote by  $\he n$  the $n$-dimensional Heisenberg
group, identified with $\rn {2n+1}$ through exponential
coordinates. A point $p\in \he n$ is denoted by
$p=(x,y,t)$, with both $x,y\in\rn{n}$
and $t\in\R$.
   If $p$ and
$p'\in \he n$,   the group operation is defined by
\begin{equation*}
p\cdot p'=(x+x', y+y', t+t' + \frac12 \sum_{j=1}^n(x_j y_{j}'- y_{j} x_{j}')).
\end{equation*}
The unit element of $\he n$ is the origin, that will be denote by $e$.
For
any $q\in\he n$, the {\it (left) translation} $\tau_q:\he n\to\he n$ is defined
as $$ p\mapsto\tau_q p:=q\cdot p. $$
The Lebesgue measure in $\mathbb R^{2n+1}$ 
is a Haar measure in $\he n$.

For a general review on Heisenberg groups and their properties, we
refer to \cite{Stein}, \cite{GromovCC} and to \cite{VarSalCou}.
We limit ourselves to fix some notations, following \cite{FSSC_advances}.

The Heisenberg group $\he n$ can be endowed with the homogeneous
norm (Cygan-Kor\'anyi norm)
\begin{equation}\label{gauge}
\varrho (p)=\big(|p'|^4+ 16\, p_{2n+1}^2\big)^{1/4},
\end{equation}
and we define the gauge distance (a true distance, see
 \cite{Stein}, p.\,638), that is left invariant i.e. $d(\tau_q p,\tau_q p')=d(p,p' )$ for all $p,p'\in\he n$)
as
\begin{equation}\label{def_distance}
d(p,q):=\varrho ({p^{-1}\cdot q}).
\end{equation}
Finally, the balls for the metric $d$ are le so-called Kor\'anyi balls
\begin{equation}\label{koranyi}
B(p,r):=\{q \in  \he n; \; d(p,q)< r\}.
\end{equation}

Notice that Kor\'anyi balls are star-shaped with respect to the origin and convex smooth sets.

A straightforward computation shows that there exists $c_0>1$ such that
\begin{equation}\label{c0}
c_0^{-2} |p| \le \rho(p) \le |p|^{1/2},
\end{equation}
provided $p$ is close to $e$.
In particular, for $r>0$ small, if we denote by $B_{\mathrm{Euc}}(e,r)$ the Euclidean ball centred at $e$ of radius $r$,
\begin{equation}\label{balls inclusion}
B_{\mathrm{Euc}}(e,r^2) \subset B(e,r) \subset B_{\mathrm{Euc}}(e, c_0^2 r).
\end{equation}

    We denote by  $\mfrak h$
 the Lie algebra of the left
invariant vector fields of $\he n$. The standard basis of $\mfrak
h$ is given, for $i=1,\dots,n$,  by
\begin{equation*}
X_i := \partial_{x_i}-\frac12 y_i \partial_{t},\quad Y_i :=
\partial_{y_i}+\frac12 x_i \partial_{t},\quad T :=
\partial_{t}.
\end{equation*}
The only non-trivial commutation  relations are $
[X_{i},Y_{i}] = T $, for $i=1,\dots,n.$ 
The {\it horizontal subspace}  $\mfrak h_1$ is the subspace of
$\mfrak h$ spanned by $X_1,\dots,X_n$ and $Y_1,\dots,Y_n$:
${ \mfrak h_1:=\mathrm{span}\,\left\{X_1,\dots,X_n,Y_1,\dots,Y_n\right\}\,.}$

\noindent Coherently, from now on, we refer to $X_1,\dots,X_n,Y_1,\dots,Y_n$
(identified with first order differential operators) as
the {\it horizontal derivatives}. Denoting  by $\mfrak h_2$ the linear span of $T$, the $2$-step
stratification of $\mfrak h$ is expressed by
\begin{equation*}
\mfrak h=\mfrak h_1\oplus \mfrak h_2.
\end{equation*}

\bigskip

The stratification of the Lie algebra $\mfrak h$ induces a family of non-isotropic dilations 
$\delta_\lambda: \he n\to\he n$, $\lambda>0$ as follows: if
$p=(x,y,t)\in \he n$, then
\begin{equation}\label{dilations}
\delta_\lambda (x,y,t) = (\lambda x, \lambda y, \lambda^2 t).
\end{equation}
Notice that the gauge norm \eqref{gauge} is positively $\delta_\lambda$-homogenous,
so that the Lebesgue measure of the ball $B(x,r)$ is $r^{2n+2}$ up to a geometric constant
(the Lebesgue measure of $B(e,1)$).

Thus, the {\sl homogeneous dimension}  of $\he n$
with respect to $\delta_\lambda$, $\lambda>0$, equals 
$$
Q:=2n+2.
$$
It is well known that the topological dimension of $\he n$ is $2n+1$,
since as a smooth manifold it coincides with $\R^{2n+1}$, whereas
the Hausdorff dimension of $(\he n,d)$ is $Q$.

The vector space $ \mfrak h$  can be
endowed with an inner product, indicated by
$\scalp{\cdot}{\cdot}{} $,  making
    $X_1,\dots,X_n$,  $Y_1,\dots,Y_n$ and $ T$ orthonormal.
    
Throughout this paper, we write also
\begin{equation}\label{campi W}
W_i:=X_i, \quad W_{i+n}:= Y_i\quad { \mathrm{and} } \quad W_{2n+1}:= T, \quad \text
{for }i =1, \dots, n.
\end{equation}

As in \cite{folland_stein},
we also adopt the following multi-index notation for higher-order derivatives. If $I =
(i_1,\dots,i_{2n+1})$ is a multi--index, we set  
\begin{equation}\label{WI}
W^I=W_1^{i_1}\cdots
W_{2n}^{i_{2n}}\;T^{i_{2n+1}}.
\end{equation}
By the Poincar\'e--Birkhoff--Witt theorem, the differential operators $W^I$ form a basis for the algebra of left invariant
differential operators in $\he n$. 
Furthermore, we set 
$$
|I|:=i_1+\cdots +i_{2n}+i_{2n+1}$$
 the order of the differential operator
$W^I$, and   
$$d(I):=i_1+\cdots +i_{2n}+2i_{2n+1}$$
 its degree of homogeneity
with respect to group dilations.

The dual space of $\mfrak h$ is denoted by $\covH 1$.  The  basis of
$\covH 1$,  dual to  the basis $\{X_1,\dots , Y_n,T\}$,  is the family of
covectors $\{dx_1,\dots, dx_{n},dy_1,\dots, dy_n,\theta\}$ where 
$$ \theta
:= dt - \frac12 \sum_{j=1}^n (x_jdy_j-y_jdx_j)$$ is  the {\it contact
form} in $\he n$. 
We denote by $\scalp{\cdot}{\cdot}{} $ the
inner product in $\covH 1$  that makes $(dx_1,\dots, dy_{n},\theta  )$ 
an orthonormal basis.

\medskip

\subsection{Sobolev spaces, distributions and kernels in $\he n$}\label{sobolev kernels}

Let  $U\subset \he n$ be an open set. We shall use the following classical notations:
$\mc E(U)$ is the space of all smooth function on $U$,
and $\mc D(U)$  is the space of all compactly supported smooth functions on $U$,
endowed with the standard topologies (see e.g. \cite{treves}).
The spaces $\mc E'(U)$ and $\mc D'(U)$ are their
dual spaces of distributions. 

Let $1\le p \le\infty$
and $m\in\mathbb N$,  $W^{m,p}_{\mathrm{Euc}}(U)$
denotes the usual Sobolev space. 

We remind also the notion of 
(integer order) Folland-Stein Sobolev space (for a general presentation, see e.g. \cite{folland} and \cite{folland_stein}).

\begin{definition}\label{integer spaces} If $U\subset \he n$ is an open set, $1\le p \le\infty$
and $m\in\mathbb N$, then
the space $W^{m,p}(U)$
is the space of all $u\in L^p(U)$ such that, with the notation of \eqref{WI},
$$
W^Iu\in L^p(U)\quad\mbox{for all multi-indices $I$ with } d(I) \le m,
$$
endowed with the natural norm  
$$\|\,u\|_{W^{k,p}(U)}
:= \sum_{d(I) \le m} \|W^I u\|_{L^p(U)}.
$$

\end{definition}

{  Folland-Stein Sobolev spaces enjoy the following properties akin to those
of the usual Euclidean Sobolev spaces (see \cite{folland}, and, e.g. \cite{FSSC_houston}).}

\begin{theorem} If $U\subset \he n$, $1\le p \le  \infty$, and $k\in\mathbb N$, then
\begin{itemize}
\item[i)] $ W^{k,p}(U)$ is a Banach space.
\end{itemize}
In addition, if $p<\infty$,
\begin{itemize}
\item[ii)] $ W^{k,p}(U)\cap C^\infty (U)$ is dense in $ W^{k,p}(U)$;
\item[iii)] if $U=\he n$, then $\mc D(\he n)$ is dense in $ W^{k,p}(U)$;
\item[iv)]  if $1<p<\infty$, then $W^{k,p}(U)$ is reflexive.
\end{itemize}

\end{theorem}

\begin{theorem}\label{donne}[see \cite{folland}, Theorem 5.15] If $p>Q$, then 
$$
W^{1,p}(\he n) \subset L^\infty (\he n)
$$
algebraically and topologically.
\end{theorem}

\begin{definition} { If $U\subset \he n$ is open} and if
$1\le p<\infty$,
we denote  by $\WO{k}{p}{U}$
the completion of $\mc D(U)$ in $W^{k,p}(U)$.
\end{definition}
{  \begin{remark}
If $U\subset \he n$ is bounded, by (iterated) Poincar\'e inequality (see e.g. \cite{jerison}), it follows that the norms
\begin{equation*}
\|u\|_{W^{k,p}(U)} \quad\mbox{and}\quad
\sum_{d(I)=k}\| W^I u\|_{ L^p(U)}
\end{equation*}
are equivalent on $\WO{k}{p}{U}$ when $1\le p<\infty$.
\end{remark}
}

Again as in \cite{folland_stein} it is possible to associate with the group structure
 a convolution (still denoted by $\ast$): if, for instance, $f\in\mc D(\he n)$ and
$g\in L^1_{\mathrm{loc}}(\he n)$, we set
\begin{equation}\label{group convolution}
f\ast g(p):=\int f(q)g(q^{-1}p)\,dq\quad\mbox{for $q\in \he n$}.
\end{equation}

\medskip

We remind that, if (say) $g$ is a smooth function and $L$
is a left invariant differential operator, then
$
L(f\ast g)= f\ast Lg.
$
We remind also that the convolution is again well defined
when $f,g\in\mc D'(\he n)$, provided at least one of them
has compact support (as customary, we denote by
$\mc E'(\he n)$ the class of compactly supported distributions
in $\he n$ identified with $\rn {2n+1}$). In this case the following identities
hold
\begin{equation}\label{convolutions var}
\Scal{f\ast g}{\phi} = \Scal{g}{\ccheck f\ast\phi}
\quad
\mbox{and}
\quad
\Scal{f\ast g}{\phi} = \Scal{f}{\phi\ast \ccheck g}
\end{equation}
 for any test function $\phi$ 
(if $f$ is a real function defined in $\he n$, we denote
    by $\ccheck f$ the function defined by $\ccheck f(p):=
    f(p^{-1})$ and, if $T\in\mc D'(\he n)$, then $\ccheck T$
    is the distribution defined by $\Scal{\ccheck T}{\phi}
    :=\Scal{T}{\ccheck\phi}$ for any test function $\phi$).

 Suppose now $f\in\mc E'(\he n)$ and $g\in\mc D'(\he n)$. Then,
 if $\psi\in\mathcal D(\he n)$, we have
 \begin{equation}\label{convolution by parts}
 \begin{split}
\Scal{(W^If)\ast g}{\psi}&=
 \Scal{W^If}{\psi\ast \ccheck g} =
  (-1)^{|I|}  \Scal{f}{\psi\ast (W^I \,\ccheck g)} \\
&=
 (-1)^{|I|} \Scal{f\ast \ccheck W^I\,\ccheck g}{\psi}.
\end{split}
\end{equation}

\medskip

\begin{proposition}\label{treves} We have:
\begin{enumerate}
\item if $\phi \in \mc D(\he n)$ and $T\in \mc D'(\he n)$, then $\phi \ast T\in \mc E(\he n)$
(see \cite{treves}, Theorem 7.23);
\item the convolution maps $\mc E(\he n)\times \mc E'(\he n)$ into $\mc E(\he n)$ (see \cite{treves}, p. 288);
\item the map $(S,T) \to S\ast T$ defined by
$$
\Scal{S\ast T}{\phi}_{\mc D', \mc D} =: \Scal{S}{\phi \ast \ccheck T}_{\mc E', \mc E} 
$$
is a separately continuous bilinear map from $\mc E'(\he n)\times \mc D'(\he n)$ to $\mc D'(\he n)$;
\item if $T\in \mc E'(\he n)$, and $P$ is a differential operator, then $PT\in \mc E'(\he n)$ and $\supp PT \subset \supp T$
(see \cite{treves}, 24.3);
\item let $U,U'\subset \he n$ be open sets, $U\subset U'$. If $T\in\mc D'(U')$, we define its restriction
$$
T_{\big|_{U}} \in \mc D'(U)
$$
in the sense of \cite{treves}, Example II pag. 245, i.e. for all $\phi\in \mc D(U)\subset \mc D(U')$ we set
$$
\Scal{T_{\big|_{U}}}{\phi} := \Scal{T}{\phi}.
$$
\item let $U,U'\subset \he n$ be open sets, $U\Subset U'$. Let $\beta, \hat{\beta}\in \mc E' (\he n)$ be such that
$\hat{\beta}_{\big|_{U'}} = {\beta}_{\big|_{U'}}$. If $k\in L^1 (\he n)$ and $ \supp k \subset B(e,R) $ with $R>0$
small enough, then 
$$
\big(\beta\ast k\big)_{\big|_{U}} = \big( \hat{\beta}\ast k \big)_{\big|_{U}}.
$$
\end{enumerate}

\end{proposition}

\begin{proof} Let us prove (6). Take $\phi\in \mc D(U)$ and assume $R< \mathrm{dist}\, (U, \partial U')$. Then
\begin{equation*}\begin{split}
&\Scal{\beta\ast k}{\phi} = \Scal{\beta}{\phi\ast\ccheck k}
\\&
\hphantom{xxx}  =\Scal{\hat\beta}{\phi\ast\ccheck k}\qquad\mbox{(since $\supp \phi\ast\ccheck k$
is contained in a $R$-neighborhood of $U$)}
\\&
\hphantom{xxx}  = \Scal{\hat\beta\ast k}{\phi}.
\end{split}\end{equation*}

\end{proof}

\begin{theorem}  \label{folland cont} We have:
\begin{itemize}
         \item[i)] Hausdorff-Young's inequality holds, i.e.,  if $f\in L^p(\he n)$, $g\in L^q(\he n)$, $1\le p, q,r \le \infty$ and $\frac1p + \frac1q - 1 = \frac1r$, then $f\ast g\in L^r(\he n)$
         (see \cite{folland_stein}, Proposition 1.18) .
	\item[ii)] If $K$ is a kernel of type $0$, $1<p<\infty$, $s\ge 0$, then the mapping $T:u\to u\ast K$ defined for $u\in\mc D(\he n)$ extends to a bounded operator on $W^{s,p}(\he n)$
	(see \cite{folland}, Theorem 4.9).
	\item[iii)]Suppose $0<\mu <Q$, $1<p<Q/\mu$ and $\frac{1}{q}=\frac{1}{p}-\frac{\mu}{Q}$. Let $K$ be a kernel of type $\mu$. If $u\in L^p(\he n)$ the convolutions $u\ast K$ and $K\ast u$ exists a.e. and are in $L^q(\he n)$ and there is a constant $C_p>0$ such that 
	$$
	\|u\ast K\|_q\le C_p\|u\|_p\quad \mathrm{and}\quad \| K\ast u\|_q\le C_p\|u\|_p\,
	$$
	(see \cite{folland}, Proposition 1.11).
	\item[iv)] Suppose  $s\ge 1$, $1<p<Q$, and let $U$ be a bounded open set. If $K$ is a kernel of type $1$ and $u\in W^{s-1,p}(\he n)$  with $\supp u \subset U$, then 
	$$
	\|u\ast K\|_{W^{s,p}(\he n)} \le C_{U} \|u\|_{W^{s-1,p}}(\he n).
	$$
\end{itemize} 

\end{theorem}

\begin{proof} The proof of iv) can be carried out relying on Theorems 4.10, 4.9 and Proposition 1.11 of \cite{folland}, keeping into account that
$L^{pQ/(Q-p)}(U) \subset L^{p}$ and  ii) above. Indeed
\begin{equation*}\begin{split}
\|u\ast K\|_{W^{s,p}(\he n)} &\le C\big\{ \|u\ast K\|_{L^p (\he n)} +\sum_{\ell =1}^m \|u\ast W_\ell K\|_{W^{s-1,p}(\he n)} \big\}
\\&
 \le C\big\{ \|u\ast K\|_{L^p (\he n)} + \|u\|_{W^{s-1,p}(\he n)} \big\}
\\&
\le 
C\big\{ \|u\|_{L^{pQ/(Q-p}(\he n)} + \|u\|_{W^{s-1,p}(\he n)} \big\}
\le
C_{U} \|u\|_{W^{s-1,p}}(\he n).
\end{split}\end{equation*}

\end{proof}

 We state now a few properties
 related to the convolution in $L^\infty$ that will be used in the sequel and
 - as far as we know - are not explicitly stated in the literature.

We stress that
we have to proceed carefully and we cannot use the corresponding results in \cite{BFP2}, \cite{BFP3}
because of the presence of the $L^\infty$-space. The
proofs follow verbatim those of analogous statements in the Euclidean setting,
keeping in mind that the group convolution is not commutative.

We remind first that, if $a\in L^\infty ({\he n}) \subset L^1_{\mathrm{loc}}$,
then the map $\phi\mapsto \int a(x) \phi(x)\, dx$ defines a distribution in $\mc D'(U)$
for all open sets $U\subset \he n$.

\begin{lemma}\label{orsay march}
Let $a\in \mc E' ({\he n}) $ (i.e. $a\in \mc D'({\he n})$ and $a$ has compact support, see \cite{treves} Theorem 24.2).
 If $\phi\in \mc D (\he n )$ and $K$ is a kernel in $L^1_{\mathrm{loc}}$, we notice first that $\phi\ast \ccheck K \in\mc E(\he n)$ so that  the map
$$
 \phi\mapsto \Scal{ a}{\phi\ast \ccheck K}_{\mc E',\mc E}
=: \Scal{ a\ast K}{\phi}_{\mc D',\mc D}
$$
belongs to $\mc D'$. 

Moreover, if $W$ is a horizontal derivative, then the convolution
$a\ast WK$ is well defined since $a$ is compactly supported. 
In addition,
\begin{equation}\label{orsay march eq:1}
W( a\ast K) = a \ast WK.
\end{equation}

\end{lemma}

\begin{proof} The first statement follows from \cite{treves}, Definition 27.3 and Theorem 27.6.

As for the last statement, consider a test function $\phi \in \mc D(\he n)$. We claim that
\begin{equation*}\begin{split}
\Scal{W( a\ast K)}{\phi} &= \Scal{ a\ast K}{W\phi}= \Scal{ a\ast WK}{\phi}.
\end{split}\end{equation*}
Indeed $\Scal{ a\ast K}{W\phi}= \Scal{ a}{W\phi\ast \ccheck K}$. Since
\begin{equation*}\begin{split}
W\phi\ast \ccheck K (x) = \int W\phi (y) K(x^{-1}y)\, dy =  \int \phi (y) (WK)(x^{-1}y)\, dy = \phi\ast \ccheck (WK) (x),
\end{split}\end{equation*}
we can conclude since
$\Scal{ a}{W\phi\ast \ccheck K} = \Scal{ a}{\phi\ast \ccheck (WK)}=\Scal{a\ast WK}{\phi}$.

\end{proof}

\begin{proposition}\label{freccetta}  Let $U$ be a bounded open subset of $\he n$, and suppose  $a\in L^\infty(U)$ is compactly supported. If  $K$ is a kernel in $L^1_{\mathrm{loc}}$, 
then the convolution $ a\ast K$ defined in Lemma \ref{orsay march} belongs to $L^\infty(U)$, and
$$
\| a\ast K\|_{L^\infty(U)} \le C (K,U, \supp a) \| a\|_{L^\infty(U)}.
$$

\end{proposition}

\begin{proof} Take $\phi\in \mc D(U)$. We note first that $\phi\ast \ccheck K$ belongs to $L^1(\supp a)$
and 
\begin{equation}\label{L1L1}
\| \phi\ast \ccheck K\|_{L^1(\supp a)} \le C(K)\|\phi\|_{L^1(U)}.
\end{equation}
Indeed, if $x\in \supp a$,
$$
|(\phi\ast \ccheck K)(x) |\le  \int_{d(z,e)\le R}| \phi(xz) |\, |K(z)|\, dz,
$$
since 
$$
d(z,e)\le d(z^{-1}, x)+ d(x, e) \le \mathrm{diam}\,  (U) + d(\supp a, e)=: R.
$$
Thus
$$
\| \phi\ast \ccheck K\|_{L^1(\supp a)} \le \|  |\phi| \ast (|K| \chi_{B(e,R)} \|_{L^1(\he n)}\le
C(K,R)\|\phi\|_{L^1(U)} .
$$
Thus, by definition, 
$$
| \Scal{ a\ast K}{\phi}_{\mc D',\mc D}|\le C (K, U, \supp a)  \|\phi\|_{L^1(U)}  \| a\|_{L^\infty(U)},
$$
and the assertion follows by duality since $L^\infty = (L^1)^*$.
\end{proof}

\subsection{Multilinear algebra in $\he n$ and Rumin's complex}\label{rumin complex}
Unfortunately, when dealing 
with differential forms in $\he n$, 
the de Rham complex lacks scale invariance under anisotropic dilations (see \eqref{dilations}). 
Thus, a substitute for de Rham's complex, that recovers scale invariance under $\delta_t$ has been defined 
by M. Rumin, \cite{rumin_jdg}. In turn, this notion makes sense for arbitrary contact manifolds. 
We refer to  \cite{rumin_jdg} and \cite{BFTT}, \cite{BFP2} for details of the construction. In the present paper, we shall merely need the following list of formal properties
(for the sake of completeness, in an Appendix we describe in more detail the construction of Rumin's complex).

\medskip

Throughout this paper, $\bigwedge^h\mathfrak h$ denotes the $h$-th exterior power of the Lie algebra $\mathfrak h$.  Keeping in mind that the Lie algebra $\mathfrak h$ can be identified with the
tangent space to $\he n$ at $x=e$ (see, e.g. \cite{GHL}, Proposition 1.72), 
starting from $\cov h$ we can define by left translation  a fiber bundle
over $\he n$  that we can still denote by $\cov h \simeq \bigwedge^hT^*\he{n}$. 
Moreover, a scalar product in $\mathfrak h$ induces a scalar product and a norm on $\bigwedge^h\mathfrak h$.

We can think of $h$-forms as sections of 
$\cov h$ and we denote by $\Omega^h$ the
vector space of all smooth $h$-forms.

\begin{itemize}
\medskip\item For $h=0,\ldots,2n+1$, the space of Rumin $h$-forms, $E_0^{h}$, is the space of smooth 
sections of a left-invariant subbundle of $\bigwedge^h\mathfrak h$ (that we still denote by $E_0^{h}$). 
Hence it inherits the inner product and the norm of $\bigwedge^h\mathfrak h$.
\medskip
\item A differential operator $d_c:E_0^{h}\to E_0^{h+1}$ is defined. It is left-invariant, 
homogeneous with respect to group dilations. It is a first order homogeneous operator
in the horizontal derivatives in degree $\neq n$, whereas {\bf it is a second
order homogeneous horizontal operator in degree $n$}. 
\medskip\item Altogether, operators $d_c$ form a complex: $d_c\circ d_c=0$.
\medskip\item This complex is homotopic to de Rham's complex $(\Omega^\bullet,d)$. 
More precisely there exist a sub-complex $(E,d)$ of the de Rham complex and a suitable ``projection'' $\Pi_E:\Omega^\bullet \to E^\bullet$
such that
 $\Pi_E$ is a differential operator of order $\le 1$ in the horizontal  derivatives.
\item $\Pi_{E}$ is a chain map, i.e.
$$
d\Pi_{E} = \Pi_{E}d.
$$
\item Let $\;\Pi_{E_0}$ be the orthogonal projection 
on $E_0^\bullet$.  Then
 $$\Pi_{E_0}\Pi_{E}\Pi_{E_0}=\Pi_{E_0}\qquad\mbox{and}\qquad \Pi_{E}\Pi_{E_0}\Pi_{E}=\Pi_{E}.$$
(we stress that $\Pi_{E_0}$ is an algebraic operator).
\item The exterior differential $d_c$ can be written as
$$d_c =\Pi_{E_0}d \Pi_E \Pi_{E_0}.$$
\end{itemize}

Let us list a bunch of notations for vector-valued function spaces (for the scalar case,
we refer to Section \ref{sobolev kernels}).

\begin{definition} \label{dual spaces forms-no} If $U\subset \he n$ is an open set, $0\le h\le 2n+1$, $1\le p\le \infty$ and $m\ge 0$,
we denote by $L^{p}(U,\cov{h})$, $\mc E (U,\cov{h})$,  $\mc D (U,\cov{h})$, 
$W^{m,p}(U,\cov{h})$ (by $\WO{m}{p}{U,\cov{h}}$)
the space of all sections of $\cov{h}$ such that their
components with respect to a given left-invariant frame  belong to the corresponding scalar spaces.

The spaces $L^{p}(U,E_0^h)$, $\mc E (U,E_0^h)$,  $\mc D (U,\E_0^h)$,  
$W^{m,p}(U,E_0^{h})$ and $\WO{m}{p}{U,E_0^h}$ are defined in the same way. 

Finally, the spaces $W^{m,p}_{\mathrm{Euc}}(U,\cov{h})$,  $\stackrel{\circ}{W}\!\!^{m,p}_{\mathrm{Euc}} (U,\cov{h})$,
 $W^{m,p}_{\mathrm{Euc}}(U,E_0^{h})$ and $\stackrel{\circ}{W}\!\!^{m,p}_{\mathrm{Euc}}(U,E_0^h)$
are defined replacing Folland-Stein Sobolev spaces by usual Sobolev spaces.

Clearly, all these definitions
are independent of the choice of frame.
\end{definition}

When $d_c$ is second order (when acting on forms of degree $n$), $(E_0^\bullet,d_c)$ stops behaving like a differential module. This is the source of many complications. In particular, the classical Leibniz formula for the de Rham complex $d(\alpha\wedge\beta)=d\alpha\wedge\beta\pm \alpha\wedge d\beta$ in general fails to hold (see \cite{BFT2}-Proposition A.7). This causes several technical difficulties when we want to localize our estimates by means of cut-off functions.

\begin{lemma} [see also \cite{BFP3}, Lemma 4.1] \label{leibniz} If $\zeta$ is a smooth real function, then the following formulae hold in the sense of distributions:
\begin{itemize}
\item[i)] if $h\neq n$, then on $E_0^h$ we have
$$
[d_c,\zeta] = P_0^h(W\zeta),
$$
where $P_0^h(W\zeta): E_0^h \to E_0^{h+1}$ is a linear homogeneous differential operator of order zero with coefficients depending
only on the horizontal derivatives of $\zeta$. If $h\neq n+1$, an analogous statement holds if we replace
$d_c$ { in degree $h$ with $d^*_c$ in degree $h+1$};
\item[ii)] if $h= n$, then on $E_0^n$ we have
$$
[d_c,\zeta] = P_1^n(W\zeta) + P_0^n(W^2\zeta) ,
$$
where $P_1^n(W\zeta):E_0^n \to E_0^{n+1}$ is a linear homogeneous differential operator of order 1 (and therefore
horizontal) with coefficients depending
only on the horizontal derivatives of $\zeta$, and where $P_0^h(W^2\zeta): E_0^n \to E_0^{n+1}$ is a linear homogeneous differential operator in
the horizontal derivatives of order 0
 with coefficients depending
only on second order horizontal derivatives of $\zeta$. If $h = n+1$, an analogous statement holds if we replace
$d_c$ { in degree $n$ with $d^*_c$ in degree $n+1$}.

\end{itemize}

\end{lemma}

\begin{remark}\label{coleibniz} 
On forms of degree $h>n$, Lemma \ref{leibniz} i) takes the following simpler form. If $\alpha\in L^1_{\mathrm{loc}} (\he n, E_0^h)$ with $h>n$ and $\psi\in\mc E(\he n)$, then
\begin{equation*}\begin{split}
d_c(\psi\alpha) = d(\psi\alpha) = d\psi\wedge\alpha + \psi d\alpha =
d_c\psi\wedge\alpha + \psi d_c\alpha,
\end{split}\end{equation*}
This follows from Theorem \ref{main rumin new}, viii), since $\alpha$ is a multiple of $\theta$. 
\end{remark}

Leibniz formula has the following quantitative form.
 
\begin{remark}\label{commutators_remark}
Denote by $B=B(e,1)$ the unit ball in $\mathbb H^n$. If $\lambda>1$,  let $\zeta$ be a smooth function on  $\mathbb{H}^n$ that is supported outside of a neighborhood
of $B$,
such that $W_i \zeta$ is compactly supported in $B_\lambda=B(e,\lambda)$ for $i=1,\dots,2n$. 
\begin{itemize}
\item[i)] If $h\neq n$, let $\sigma \in L^\infty\left( B_\lambda, E_0^h\right)\cap d^{-1}L^\infty\left( B_\lambda, E_0^h\right)$  
 , then 
\begin{equation}\label{maggio 8 3}
\|d_c(\zeta\sigma)\|_{L^\infty(B_\lambda, E_0^h)}
 \leq C_\zeta \Big(\|\sigma\|_{L^\infty(B_\lambda, E_0^h)}+\|d_c\sigma\|_{L^\infty(B_\lambda, E_0^{h+1})}\Big).
\end{equation}
\item[ii)]  If $h= n$ let $\sigma \in W^{1,\infty}\left( B_\lambda, E_0^n\right)\cap d^{-1}L^\infty\left( B_\lambda, E_0^n\right)$  
 and, then
\begin{equation}\label{maggio 11 1}
\|d_c(\zeta\sigma)\|_{L^\infty (B_\lambda, E_0^{n+1})}
 \leq C_\zeta \Big(\|\sigma\|_{W^{1,\infty}(B_\lambda, E_0^n)}+
 \|d_c\sigma\|_{L^\infty(S_{\zeta}, E_0^{n+1})}\Big),
\end{equation}
where $S_{\zeta}$ is a neighborhood in $B_\lambda$ of $ \supp\zeta\cap B_\lambda$ (contained in  $B_\lambda\setminus \overline{B})$.
\end{itemize}
\end{remark}

The following generalizes Remark 2.16 of \cite{BFTT}. The proof uses a notation from the Appendix, Theorem \ref{main rumin new}.

\begin{lemma}\label{byparts}
Let $U\subset \he n$ be an open set. 

\begin{itemize}
  \item[i)] Let $\psi$ be an $h$-form in $L^1_{\mathrm{loc}} (U, E_0^h)$ and $\alpha\in\mc D(U,E_0^{2n-h})$. Then
$$
\int_{U}(d_c\phi)\wedge\alpha=(-1)^{h+1} \int_{U}\phi\wedge d_c\alpha,
$$
where the left-hand side is understood in distribution sense.
\end{itemize}
Assume further that $U$ is contractible. Let $\omega$ and $\psi$ be $d_c$-closed 
Rumin forms on $U$ of complementary degrees $h$ and $2n+1-h$, with $1\leq h\leq 2n$. Then
$$
\int_{U}\omega\wedge\psi=0
$$
in the following cases:
\begin{itemize}
\item[ii)] $\omega\in L^1_{\mathrm{loc}} (U, E_0^h)$ and $\psi\in\mc D(U,E_0^{2n+1-h})$.
\item[iii)] $1<p<\infty$, $\frac1p+\frac1{p'}=1$, $\omega\in L^p_{\mathrm{loc}}(U, E_0^h)$  and $\psi\in L^{p'} (U, E_0^{2n+1-h})$ is
compactly supported in $U$.
\end{itemize}
\end{lemma}

\begin{proof}
Assume first that $\phi$ is smooth. Since $\phi\wedge\alpha$ has degree $2n\geq n+1$, $d(\phi\wedge\alpha)=d_c(\phi\wedge\alpha)$. If $h\not=n$, the formula 
$$
d_c(\phi\wedge\alpha)=(d_c\phi)\wedge\alpha+(-1)^h \phi\wedge d_c\alpha
$$
is established in \cite{PT}, Prop. 4.2. Let us assume that both $\phi$ and $\alpha$ have degree $n$. Then, by definition, $d_c\phi=d\Pi_E\phi$ where $\Pi_E\phi-\phi$ has weight $n+1$ (see Theorem \ref{main rumin new} ix)). Since $d_c\alpha$ has weight $n+2$, $(\Pi_E\phi-\phi)\wedge d_c\alpha=0$. Symmetrically, $(\Pi_E\alpha-\alpha)\wedge d_c\phi=0$. It follows that
$$
d(\Pi_E \phi\wedge\Pi_E \alpha)=(d_c\phi)\wedge\alpha+(-1)^n \phi\wedge d_c\alpha.
$$
In all cases, we have come up with a compactly supported primitive of $(d_c\phi)\wedge\alpha+(-1)^h \phi\wedge d_c\alpha$, hence
$$
\int_{U}((d_c\phi)\wedge\alpha+(-1)^h \phi\wedge d_c\alpha)=0.
$$
Formula i) extends to Rumin forms $\phi$ with distributional coefficients, and in particular to forms in $L^1_{\mathrm{loc}} (U, E_0^h)$.

\medskip

Assume first that $\omega$ and $\psi$ are smooth. 
Since Rumin's complex is homotopic to de Rham's, $\omega$ admits a smooth primitive $\phi$ on $U$, $d_c\phi=\omega$. Then i) implies that
$$
\int_{U}\omega\wedge\psi=\int_{U}(d_c\phi)\wedge\psi=\pm\int_{U}\phi\wedge d_c\psi=0.
$$
\begin{itemize}
\item[ii)] By right convolution (which commutes with the left-invariant operator $d_c$), closed forms are dense in $L^1_{\mathrm{loc}}$ $d_c$-closed forms, so the identity extends to the case where $\omega\in L^1_{\mathrm{loc}} (U, E_0^h)$.
\item[iii)] Again by right convolution, smooth $d_c$-closed forms are dense in $L^{p}_{\mathrm{loc}} (U, E_0^{h})$ $d_c$-closed forms and smooth compactly supported $d_c$-closed forms are dense in compactly supported $L^{p'} (U, E_0^{2n+1-h})$ $d_c$-closed forms. 
\end{itemize}
\end{proof}

\subsection{Rumin's Laplacian}\label{sect rumin laplacian}

 \begin{definition}\label{rumin laplacian} 
In $\he n$, following \cite{rumin_jdg}, we define
the operators $\Delta_{\he{},h}$  on $E_0^h$ by setting
\begin{equation*}
\Delta_{\he{},h}=
\left\{
  \begin{array}{lcl}
     d_cd^*_c+d^*_c d_c\quad &\mbox{if } & h\neq n, n+1;
     \\ (d_cd^*_c)^2 +d^*_cd_c\quad& \mbox{if } & h=n;
     \\d_cd^*_c+(d^*_c d_c)^2 \quad &\mbox{if }  & h=n+1.
  \end{array}
\right.
\end{equation*}

\end{definition}

Notice that $-\Delta_{\he{},0} = \sum_{j=1}^{2n}(W_j^2)$ is the usual sub-Laplacian of
$\he n$. 

 For sake of simplicity, once a basis  of $E_0^h$
is fixed, the operator $\Delta_{\he{},h}$ can be identified with a matrix-valued map, still denoted
by $\Delta_{\he{},h}$,
\begin{equation}\label{matrix form}
\Delta_{\he{},h} = (\Delta_{\he{},h}^{ij})_{i,j=1,\dots,N_h}: \mc D'(\he{n}, \rn{N_h})\to \mc D'(\he{n}, \rn{N_h}),
\end{equation}
where $\mc D'(\he{n}, \rn{N_h})$ is the space of vector-valued distributions on $\he n$, and $N_h$ is
the dimension of $E_0^h$ (see \cite{BBF}).

This identification makes possible to avoid the notion of currents: we refer to \cite{BFTT} for
a more elegant presentation.

Combining \cite{rumin_jdg}, Section 3,   and \cite{BFT3}, Theorems 3.1 and 4.1, we obtain the following result.

\begin{theorem}[see \cite{BFT3}, Theorem 3.1] \label{global solution}
If $0\le h\le 2n+1$, then the differential operator $\Delta_{\he{},h}$ is
homogeneous of degree $\mu$ with respect to group dilations, where $\mu=2$ if $h\neq n, n+1$ and  $\mu=4$ 
if $h=n, n+1$. It follows that
\begin{enumerate}
\item[i)] for $j=1,\dots,N_h$ there exists
\begin{equation}\label{numero}
    K_j =
\big(K_{1j},\dots, K_{N_h j}\big), \quad j=1,\dots N_h
\end{equation}
 with $K_{ij}\in\mc D'(\he{n})\cap \mc
E(\he{n} \setminus\{0\})$,
$i,j =1,\dots,N_h$;
\item[ii)] if $\mu<Q$, then the $K_{ij}$'s are
kernels of type $\mu$
 for
$i,j
=1,\dots, N_h$

 If $\mu=Q$,
then the $K_{ij}$'s satisfy the logarithmic estimate
$|K_{ij}(p)|\le C(1+|\ln\rho(p)|)$ and hence
belong to $L^1_{\mathrm{loc}}(\he{n})$.
Moreover, their horizontal derivatives  $W_\ell K_{ij}$,
$\ell=1,\dots,2n$, are
kernels of type $Q-1$;
\item[iii)] when $\alpha\in
\mc D(\he{n},\rn {N_h})$,
if we set
\begin{equation}\label{numero2}
   \Delta_{\he{},h}^{-1}\alpha:=
\big(    
    \sum_{j}\alpha_j\ast  K_{1j},\dots,
     \sum_{j}\alpha_j\ast  K_{N_hj}\big),
\end{equation}
 then $ \Delta_{h}\Delta_{\he{},h}^{-1}\alpha =  \alpha. $
Moreover, if $\mu<Q$, also $\Delta_{\he{},h}^{-1}\Delta_{h} \alpha =\alpha$.

\item[iv)] if $\mu=Q$, then for any $\alpha\in
\mc D(\he{n},\rn {N_h})$ there exists 
$\beta_\alpha:=(\beta_1,\dots,\beta_{N_h})\in \rn{N_h}$,  such that
$$\Delta_{\he{},h}^{-1}\Delta_{h}\alpha - \alpha = \beta_\alpha.$$

\item[iv)] $\Delta_{\he{},h}^{-1}: \mc D(\he{n},\rn {N_h}) \to \mc E(\he{n},\rn {N_h})$
and  $\Delta_{\he{},h}^{-1}: \mc E'(\he{n},\rn {N_h}) \to \mc D'(\he{n},\rn {N_h})$.

\end{enumerate}
\end{theorem}

{   \begin{remark}\label{pavidi}
If $\mu<Q$, 
 $ \Delta_{\he{},h} (\Delta_{\he{},h}^{-1} - \ccheck \Delta_{\he{},h}^{-1}) = 0$ and hence $\Delta_{\he{},h}^{-1} = \ccheck \Delta_{\he{},h}^{-1}$,
by the Liouville-type theorem of \cite{BFT3}, Proposition 3.2.
\end{remark}

\begin{remark} From now on, if there are no possible misunderstandings, we identify $\Delta_{\he{},h}^{-1}$ with its kernel.
\end{remark}
 }

\begin{lemma}[see \cite{BFP2}, Lemma 4.11]\label{comm} 
If $\phi\in\mc D(\he n, E_0^h)$ and $n\ge 1$, then
\begin{itemize}
\item[i)]$
d_c \Delta^{-1}_{\mathbb H, h}\phi = \Delta^{-1}_{\mathbb H, h+1} d_c\phi$, \qquad $h=0,1,\dots, 2n$, 
\qquad $h\neq n-1, n+1$.

\item[ii)] $d_c \Delta^{-1}_{\mathbb H, n-1}\phi = d_cd^*_c\Delta^{-1}_{\mathbb H, n} d_c\phi$ \qquad ($h=n-1$).

\item[iii)]$
d_cd^*_c d_c \Delta^{-1}_{\mathbb H, n+1}\phi = \Delta^{-1}_{\mathbb H, n+2} d_c\phi$,
 \qquad ($h=n+1$).

\item[iv)]$d^*_c \Delta^{-1}_{\mathbb H, h}\phi = \Delta^{-1}_{\mathbb H, h-1} d^*_c\phi$
 \qquad $ h=1,\dots, 2n+1$, \qquad $h\neq n, n+2$.
 
 \item[v)] $d^*_c \Delta^{-1}_{\mathbb H, n+2}\phi = d^*_c d_c\Delta^{-1}_{\mathbb H, n+1} 
 d^*_c\phi$ \qquad ($h=n+2$).

\item[vi)]$
d^*_c d_c d^*_c  \Delta^{-1}_{\mathbb H, n}\phi = \Delta^{-1}_{\mathbb H, n-1} d^*_c \phi$,
 \qquad ($h=n$).
\end{itemize}
\end{lemma}

\section{Dual formulations and proof of Theorem \ref{poincare infty}}\label{caldo}

The interior Poincar\'e inequality of Theorem \ref{poincare infty} relies on three
tools:
\begin{itemize}
\item[i)] the formulation by duality of the interior Poincar\'e inequality $\he{}$-$\mathrm{Poincar\acute{e}}_{p, \infty}(h)$
(see Proposition \ref{poincaredual}) below;
\item[ii)] the integration by parts formula of Lemma \ref{byparts};
\item[iii)] the relationship between  
$\he{}$-$\mathrm{Sobolev}_{ 1,p'}(2n+2-h) $ and $\he{}$-$\mathrm{Poincar\acute e}_{p,\infty}(h) $,
when $1<p<\infty$ and $p'$ is the dual exponent of $p$.
\end{itemize}

\begin{proposition}
\label{poincaredual}
Assume that $1 \le  h <2n+1$. Let $B$, $B_\lambda$ be concentric balls as in Definition \ref{poincare def}. Take $1<p<\infty$. Then the $\he{}$-$\mathrm{Poincar\acute{e}}_{p, \infty}(h)$
 inequality in $E_0^h$ holds if and only if there exists a constant $C$ such that for every $d_c$-closed differential $h$-form $\omega$ on $L^p(B_\lambda,E_0^h)$ and every 
 smooth differential $(2n+1-h)$-form $\alpha$ with compact support in $B$,
\begin{equation}\label{dual 3/4}
|\int_{B}\omega\wedge\alpha|\leq C\,\|\omega\|_{L^p(B_\lambda,E_0^h)}\|d_c\alpha\|_{L^{1}(B,E_0^{2n+2-h})}.
\end{equation}
An analogous statement holds for the global Poincar\'e inequality on $\he n$.
\end{proposition}

\begin{proof}
Suppose $\omega$ is a $h$-form satisfying Definition \ref{poincare def}.
If $\omega_{|B}=d_c\phi$ with $\phi\in L^\infty(\he n, E_0^{h-1})$ as in Definition \ref{poincare def}, and $\alpha \in \mathcal D(B, E_0^{2n+1-h})$, then, according to Lemma \ref{byparts} i),
\begin{align*}
|\int_{B}\omega\wedge\alpha|
&=|\int_{B}\phi\wedge d_c\alpha|\\
&\leq\|\phi\|_{L^{\infty}(B,E^{h-1})}\|d_c\alpha\|_{L^{1}(B, E_0^{2n+2-h})}.
\end{align*}
Since $\|\phi\|_{L^{\infty}(B,E_0^{h-1})}$ can be estimated by $\|\omega\|_{L^{p}(B_\lambda,E_0^h)}$, inequality
\eqref{dual 3/4} follows.

Conversely, assume that for all forms $\alpha\in \mathcal D(B;E_0^{2n+1-h})$, 
\begin{eqnarray*}
|\int_{B}\omega\wedge\alpha|\leq C\,\|\omega\|_{L^{p}(B_\lambda;E_0^h)}\|d_c\alpha\|_{L^{1}(B;
E_0^{2n+2-h})}.
\end{eqnarray*}
Define a linear functional $\eta$ on differentials of smooth $(2n+1-k)$-forms with compact support in $B$ as follows. If $\beta=d_c\alpha$, $\alpha\in\mathcal D(B; E_0^{2n+1-h})$, set
\begin{eqnarray*}
\eta(\beta)=\int_{B}\omega\wedge\alpha.
\end{eqnarray*}
Then $\eta$ is well defined since $\omega\in L^1_{\mathrm{loc}}$, and is continuous in $L^{1}(B;E_0^{2n+2-h})$-norm, by \eqref{dual 3/4}, i.e.
\begin{eqnarray*}
|\eta(\beta)|\leq C\,\|\omega\|_{L^{p}(B_\lambda;E_0^{h})}\|\beta\|_{L^{1}(B;E_0^{2n+2-h})}.
\end{eqnarray*}
By the Hahn-Banach theorem, $\eta$ extends to a linear functional on all of $L^{1}(B; E_0^{2n+2-h})$, with the same norm. Such a functional is represented by a differential form $\phi\in 
L^\infty(B;E_0^{h-1})$ as follows,
\begin{eqnarray*}
\eta(\beta)=\int_{B}\phi\wedge\beta.
\end{eqnarray*}
The $L^\infty$-norm of $\phi$ is at most $C\,\|\omega\|_{L^{p}(B_\lambda;E_0^{h})}$.
Since, for all forms $\alpha\in \mathcal D(B; E_0^{2n+1-h})$, 
\begin{eqnarray*}
\int_{B}\phi\wedge d_c\alpha=\eta(d_c\alpha)=\int_{B}\omega\wedge\alpha,
\end{eqnarray*}
one concludes that $d_c\phi=\omega$ on $B$ in the distributional sense.

The proof of the statement for the global Poincar\'e inequality
can be carried out repeating verbatim the same arguments.
\end{proof}

\begin{proposition}
\label{sobolev implies poincare bis}   
 If $2\le h \le 2n+1$, $1<p<\infty$ and $p'$ is the dual exponent of $p$, then  
the interior $\he{}$-$\mathrm{Sobolev}_{1,p'}(2n+2-h) $ inequality implies the interior $\he{}$-$\mathrm{Poincar\acute e}_{p,\infty}(h) $ inequality.

An analogous statement holds for global inequalities on $\he n$.
\end{proposition}

\begin{proof} To prove the assertion, we argue by duality relying on Proposition
\ref{poincaredual}. Let $B$ and $B_\lambda$ be concentric balls, 
such that  $\he{}$-$\mathrm{Sobolev}_{1,p'}(2n+2-h) $ inequality holds in $B,B_\lambda$. Take a $d_c$-closed $h$-form $\omega$ in $L^p(B_\lambda,E_0^h)$ and 
an arbitrary smooth differential $(2n+1-h)$-form $\alpha$ with compact support in $B$. 
By Sobolev inequality, there exists 
a compactly supported differential $(2n+1-h)$-form $\beta\in L^{p'}(B_\lambda,E_0^{2n+1-h})$ such that $d_c\beta=d_c\alpha$ in $B$ and
\begin{eqnarray}\label{H Sobolev 6 apr bis}
\|\beta\|_{L^{p'}(B_\lambda,E_0^{2n+1-h})}\leq C\,\|d_c\alpha\|_{L^{1}(B,E_0^{2n+2-h})}.
\end{eqnarray}
If $h=2n+1$, $\alpha$ and $\beta$ are compactly supported functions and $d_c(\beta-\alpha)=0$, hence $\beta=\alpha$. Otherwise, since $\psi=\beta-\alpha\in L^{p'}(B_\lambda,E_0^{2n+1-h})$ is  $d_c$-closed, Lemma \ref{byparts} iii) implies that
\begin{equation*}\begin{split}
\int_{B_\lambda}\omega\wedge(\beta-\alpha) = 0.
\end{split}\end{equation*}
Therefore in both cases, by H\"older inequality and by \eqref{H Sobolev 6 apr bis},
\begin{equation*}\begin{split}
|\int_{B}\omega\wedge\alpha|=|\int_{B_\lambda}\omega\wedge\alpha|=|\int_{B_\lambda}\omega\wedge\beta|
\leq C\,\|\omega\|_{L^p(B_\lambda,E_0^h)}\|d_c\alpha\|_{L^{1}(B,E_0^{2n+2-h})}.
\end{split}\end{equation*}
By Proposition
\ref{poincaredual}, this implies $\he{}$-$\mathrm{Poincar\acute e}_{p,\infty}(h) $.

The global Sobolev inequality implies the global Poincar\'e inequality in the same manner.

\end{proof}

\begin{proof}[{\bf Proof of Theorem \ref{poincare infty}}]
Theorem \ref{poincare infty} follows straightforwardly,  combining the $L^1$ Sobolev inequality proven in \cite{BFP3}, Corollary 6.5, and previous
Proposition \ref{sobolev implies poincare bis}.

\end{proof}

 \section{Homotopy formul\ae\, in $L^\infty$}\label{section homotopy infty}

The following global homotopy formula has been proven in $\mc D(\he n, E_0^\bullet)$ in
 \cite{BFP2}, Proposition 6.9. However we have to stress that here we deal with $L^\infty$ forms, hence we have to adapt the proof since we cannot rely on a density argument in $L^\infty$.

\begin{proposition}\label{homotopy formulas L infty} If $\alpha\in L^\infty (\he n, E_0^h)$ is compactly supported,
 then the following homotopy formulas hold:
 there exist operators $K_1,K_2$
 such that
\begin{itemize}
\item[i)] if $h\neq n, n+1$, then $\alpha = d_c K_1 \alpha +  K_1d_c \alpha $ in the sense of distributions, where $K_1$ is  associated with a kernel $k_1$ of type 1;
\item[ii)]  if $h = n$, then $\alpha = d_c K_1 \alpha +  K_2 d_c\alpha$ in the sense of distributions, where $K_1$ and $ K_2$ are  associated with kernels $k_1,  k_2$ of type 1 and 2,
respectively;
\item[iii)]  if $h = n+1 $, then $\alpha = d_c K_2 \alpha +  K_1 d_c\alpha$ in the sense of distributions, where $K_2$ and $ K_1$ are  associated with kernels $k_2,  k_1$ of type 2
and $1$, respectively.
\end{itemize}

\end{proposition}

\begin{proof} The proof can be carried out by duality. Consider for instance the case i), and let $\phi\in \mc D(\he n, E_0^h)$ be a test form. 
By  Theorem \ref{global solution},  
\begin{equation*}\begin{split}
\Scal{\alpha}{\phi}_{\mc D',\mc D} & = \Scal{\alpha}{(d_c d_c^*+ d_c^* d_c )\Delta_{\mathbb H}^{-1}\phi}_{\mc D',\mc D}=\Scal{\alpha}{(d_c d_c^*+ d_c^* d_c )\Delta_{\mathbb H}^{-1}\phi}_{\mc E',\mc E}
\\&
= \Scal{\alpha}{\Delta_{\mathbb H}^{-1}d_c d_c^* \phi}_{\mc E',\mc E}
+ \Scal{\alpha}{ d_c^*\Delta_{\mathbb H}^{-1}d_c\phi}_{\mc E',\mc E}
\qquad\mbox{(by Lemma \ref{comm})}
\\&
= \Scal{ \Delta_{\mathbb H}^{-1} \alpha}{d_cd_c^* \phi}_{\mc D',\mc D}
+ \Scal{d_c^* \Delta_{\mathbb H}^{-1}d_c\alpha}{\phi}_{\mc D',\mc D}
\qquad\mbox{(since $d_c\alpha\in\mc E'$) }
\\&
= \Scal{ d_c d_c^*\Delta_{\mathbb H}^{-1} \alpha}{ \phi}_{\mc D',\mc D}
+ \Scal{d_c^* \Delta_{\mathbb H}^{-1}d_c\alpha}{\phi}_{\mc D',\mc D}
\\&
=:\Scal{d_c K_1 \alpha +  K_1d_c\alpha }{\phi}_{\mc D',\mc D},
\end{split}\end{equation*}
and the assertion follows since $K_1:=d_c^* \Delta_{\mathbb H}^{-1}$ is a kernel of type 1
by Theorem \ref{global solution}. 

The proofs of ii) and iii) can be carried out through similar duality arguments
keeping in mind  \cite{BFP2}, Proposition 6.9.

\end{proof}

\begin{remark}\label{may 21} To avoid cumbersome notations, from now on we denote by $K_0$ one of the convolution 
operators $K_1$, $K_2$, so that the homotopy formulas of Proposition \ref{homotopy formulas L infty}
can be written concisely as follows: if $\alpha\in L^\infty (\he n, E_0^h)$ is compactly supported
then $K_0 \alpha$ and $K_0 d_c\alpha$ are well defined
distributions by Lemma \ref{orsay march} and
\begin{equation}\label{may 21 eq:1}
\alpha = d_c K_0 \alpha +  K_0d_c \alpha,
\end{equation}
where $K_0$ is associated with a kernel of type 1 or 2, depending on the degree of $\alpha$.
Notice that in any case, $K_0$ belongs to $L^1_{\mathrm{loc}}$.

\end{remark}

\begin{proposition}\label{smoothing-A}
Let $U\Subset U'$ be bounded open sets in $\he n$. For $h=1,\ldots,2n$, for every $s\in\N$, there exist  a bounded operator 
\begin{equation}\label{consiglio}
	T:L^\infty(U',E_0^\bullet) \to L^\infty(U,E_0^{\bullet-1})
\end{equation}
and
a smoothing  
operator 
\begin{equation}\label{dec 29 1}
S:L^\infty(U',E_0^\bullet)  \to \mc E(U,E_0^{\bullet-1})
\end{equation}
such that, in addition,
\begin{equation}\label{smoothing eq:1A}
S \in \mc L\big(L^\infty(U',E_0^\bullet) , W^{s,\infty}(U,E_0^{\bullet-1})\big)
\end{equation}
so that for any $h$-forms $\alpha$ in $L^\infty (U',E_0^\bullet)$   such that $d_c\alpha\in L^\infty (U',E_0^{\bullet+1})$  
  the following approximate homotopy formula holds in the sense of distributions
\begin{equation}\label{homotopy dec 27}
\alpha=d_cT\alpha+T d_c \alpha+S\alpha\qquad \text{on }U.
\end{equation}

 In addition, on forms of degree $n+1$, $T$ is bounded 
\begin{equation}\label{consiglio2}
T: L^{\infty}(U',E_0^{n+1})
 \to W^{1, \infty}(U,E_0^n).
\end{equation}

Finally, if $\alpha\in L^\infty(U',E_0^\bullet)\cap d_c^{-1}(L^\infty(U',E_0^{\bullet+1}))$ we notice that,  by difference,
$$
d_c T\alpha\in L^\infty (U, E_0^h) .
$$
\end{proposition}

 \begin{proof}
 
If $\alpha \in L^\infty (U', E_0^h)$, we set $\alpha_0$ to be $\alpha$  continued by zero outside $U'$, the so-called {\sl trivial extension} of $\alpha$.
 Obviously, $\alpha_0$ belongs to $L^\infty (\he n,E_0^h)$ and is compactly supported, hence belongs to $ \mc{E'}(\he n, E_0^h)$. The trivial extension defines a continuous linear map from $\mc D(U')$ to $\mc D(U)$.
 
 Denote by $k_0$  the kernel associated with $K_0$ as defined in Remark \ref{may 21}. 
 We consider a cut-off function $\psi_R$ supported in a $R$-neighborhood
of the origin, such that $\psi_R\equiv 1$ near the origin. Then we have  
$k_0=k_0\psi_R + (1-\psi_R)k_0$ Let us denote by $K_{0,R}$
 the convolution operator associated with $k_{0,R}:=\psi_R k_0$ and by $K'_{0,R}=K_0-K_{0,R}$
 the convolution operator
associated with the kernel $k'_{0,R}:= k_0 - k_{0,R}$.
 
 The kernel $\psi_R k_0\in L^1(\he n)$, so that,
by Theorem \ref{folland cont}, i), $K_{0,R}$ maps  $L^\infty$ to $L^\infty$.

Let us apply Proposition \ref{homotopy formulas L infty} using the decomposition $K_0=K_{0,R}+K'_{0,R}$: for $\phi\in\mc D(\he n,  E_0^h)$, 
\begin{equation*}\begin{split}
\Scal{\alpha_0}{\phi}_{\mc D',\mc D} 
=\Scal{d_c K_{0,R} \alpha_0 +  K_{0,R}d_c\alpha_0 +d_c K'_{0,R} \alpha_0 +  K'_{0,R}d_c\alpha_0}{\phi}_{\mc D',\mc D},
\end{split}\end{equation*}
i.e.
$$
\alpha_0 = d_c K_{0,R} \alpha_0 +  K_{0,R}d_c\alpha_0 +d_c K'_{0,R} \alpha_0 +  K'_{0,R}d_c\alpha_0
$$
in the sense of distributions. Taking the restriction to $U$, we get
\begin{equation}\label{19 dec 1 prima}
\alpha = \big(d_c {K_{0,R}\alpha_0}\big)_{\big|_{U}} + \big({K_{0,R}d_c \alpha_0}\big)_{\big|_{U}} +
\big({d_c K'_{0,R} \alpha_0 +  K'_{0,R}d_c\alpha_0}\big)_{\big|_{U}},
\end{equation}
where the restriction has to be meant as restriction of a distribution as in Proposition \ref{treves}, (5).
First we notice that, by Proposition \ref{treves}, ii) and iv), we have
$$
S\alpha:=  \big({d_c K'_{0,R} \alpha_0 +  K'_{0,R}d_c\alpha_0}\big)_{\big|_{U}} \in \mc E(U),
$$
yielding \eqref{dec 29 1}.
Since derivatives commute with restriction, if $R>0$ is small enough, we have
$$
\big(d_c {K_{0,R}\alpha_0}\big)_{\big|_{U}} =  d_c\big( {K_{0,R}\alpha_0}\big)_{\big|_{U}},
$$
and \eqref{19 dec 1 prima} reads
\begin{equation}\label{19 dec 1}
\alpha_{\big|_{U}} = d_c\big({K_{0,R}\alpha_0}\big)_{\big|_{U}} + \big({K_{0,R}d_c \alpha_0}\big)_{\big|_{U}} +
S\alpha.
\end{equation}
If now $\beta \in L^\infty(U',  E_0^\bullet) $,   we set 
$$
T\beta := {K_{0,R}\beta_0}_{\big|_{U}}.
$$
Thus, in \eqref{19 dec 1}, we have
$$
d_c\big({K_{0,R}\alpha_0}\big)_{\big|_{U}}  = d_cT\alpha.
$$
Consider now in \eqref{19 dec 1} the term $\big({K_{0,R}d_c \alpha_0}\big)_{\big|_{U}}$. We
observe preliminarily that 
\begin{equation}\label{26 dic 1}
(d_c \alpha_0)_{\big|_{U'}} = \big((d_c\alpha)_0\big)_{\big|_{U'}},
\end{equation}
where, as above, $(d_c\alpha)_0$ is the trivial extension of $d_c\alpha$.
Indeed, if $\phi\in \mc D(U',  E_0^{\bullet+1})$ then
\begin{equation*}\begin{split}
&
\Scal{d_c \alpha_0}{\phi}_{\mc D',\mc D} = \Scal{\alpha_0}{d_c^*\phi}_{\mc D',\mc D}
= \int_{U'} \scal{\alpha}{d_c^*\phi}\, dx
= \int_{U'} \scal{d_c\alpha}{\phi}\, dx
\\&\hphantom{xxx}
= \int_{\he n} \scal{(d_c\alpha)_0}{\phi}\, dx
= \Scal{(d_c \alpha)_0}{\phi}_{\mc D',\mc D}.
\end{split}\end{equation*}
This proves \eqref{26 dic 1}. Thus, we can apply Proposition \ref{treves}, (6) and we get, for $R$ small enough,
\begin{equation*}\begin{split}
\big(K_{0,R} & d_c \alpha_0\big)_{\big|_{U}} = \big( (d_c \alpha_0) \ast k_{0,R} \big)_{\big|_{U}}
\\&
 = \big((d_c \alpha)_0 \ast k_{0,R} \big)_{\big|_{U}} = \big(K_{0,R}(d_c \alpha)_0\big)_{\big|_{U}} =T(d_c \alpha).
\end{split}\end{equation*}
Eventually,  identity \eqref{19 dec 1} becomes
$$
\alpha = d_c T\alpha + Td_c \alpha + S\alpha \qquad\text{in }U.
$$
This proves the homotopy formula \eqref{homotopy dec 27}.

Since the kernel $k_{0,R}$ belongs to $L^1(\he n)$, by Hausdorff-Young inequality
(see Theorem \ref{folland cont}, i), 
$$T: L^\infty(U',  E_0^\bullet) \to L^\infty(U,  E_0^{\bullet-1}),
$$ and
this proves \eqref{consiglio}.

Let us prove the continuity estimates \eqref{smoothing eq:1A} for the operator $S$. Consider first
the term
$$
 \big(d_c (K'_{0,R} \alpha_0)  \big)_{\big|_{U}} =  \big(d_c ( \alpha_0\ast k'_{0,R}) \big)_{\big|_{U}} .
$$
If $1\le h\le 2n$, let $(\xi_1^h,\dots, \xi^h_{\mathrm{dim}\, E_0^h})$ be a basis of $E_0^h$. Then
$\alpha= \sum_j \alpha_j   \xi_j^h$ with $\alpha_j\in L^\infty(U')$, $j=1,\dots, \mathrm{dim}\, E_0^h$.
Obviously, $\alpha_0= \sum_j (\alpha_j )_0   \xi_j^h$, and $d_c ( \alpha_0\ast k'_{0,R})$ can be
written as sum of terms of the form
$$
W^I \big((\alpha_j)_0 \ast  \kappa\big) = (\alpha_j)_0 \ast W^I \kappa, 
$$
where $\kappa $ is a  smooth kernel
and $d(I)=1$ or  $d(I)=2$, according to the degree $h$.
Thus, in order to prove \eqref{smoothing eq:1A} we have 
to estimate the $L^\infty$-norms in $U$ of a sum of terms of the form
$$
 (\alpha_j)_0 \ast W^J\kappa, 
$$
with $d(J)=s+1$ or $d(J)=s+2$, according to the degree $h$.
Then the assertion follows by Proposition \ref{freccetta}, since 
the smooth kernel $W^J\kappa$ belongs to $L^1_{\mathrm{loc}}(\he n)$
(notice that $\|\alpha_j\|_{L^\infty(U)} =  \|(\alpha_j)_0\|_{L^\infty(\he n)}$).

Analogously, if we aim to estimate the term $K'_{0,R}d_c\alpha_0 = (d_c\alpha_0)\ast k'_{0,R}$,
we have to estimate in $L^\infty(U)$ a sum of terms of the form (keep in mind \eqref{convolution by parts})
\begin{equation*}\begin{split}
(W^I (\alpha_j)_0) \ast W^J\kappa =  (\alpha_j)_0) \ast \ccheck W^I\ccheck W^J\kappa,
\end{split}\end{equation*}
where $\kappa$ is a smooth kernel, $d(J)=s$ and $d(I)=1$ or $d(I)=2$ according to the degree $h$.
Since $\ccheck W^I\ccheck W^J\kappa$ is still a smooth kernel, the estimate can be carried
out as for the first term.

Thus we are left with the proof of \eqref{consiglio2}.  To this end, we notice first that on forms of degree $h = n+1$,
 the kernel of $K_{0,R}$ is obtained by truncation near the origin
of a kernel of type 2. Therefore, on forms of degree $h= n+1$  
 all the horizontal derivatives $W K_{0,R}$  belongs to $L^1$
 and, if $\alpha\in L^\infty(U',E_0^{n+1})$, then
 \begin{equation*}\begin{split}
\| WT\alpha &  \|_{L^\infty(U,E_0^{n})} = \| W(\alpha\ast k_{0,R})   \|_{L^\infty(U,E_0^{n})} 
\\&
=  \| \alpha\ast Wk_{0,R}   \|_{L^\infty(U,E_0^{n})},
\end{split}\end{equation*}
and the proof can be carried out again by Proposition \ref{freccetta}.

\end{proof}

\begin{remark}\label{smoothing remark}
Since, in $U'$, $W^{2s,p} \subset W^{s,p}_{\mathrm{Euc}}$ for $1\le p\le\infty$, then
\eqref{smoothing eq:1A} can be equivalently stated as
\begin{equation}\label{smoothing eq:1A bis}
S \in \mc L\big(L^\infty(U',E_0^\bullet) , W^{s,\infty}_{\mathrm{Euc}}(U,E_0^{\bullet-1})\big).
\end{equation}

\end{remark}

\section{Intermediate tools: interior
$\he{}$-$\mathrm{Poincar\acute{e}}_{\infty, \infty}$ and $\he{}$-$\mathrm{Sobolev}_{\infty, \infty}$
inequalities}
\label{poincare}

In \cite{IL}, starting from Cartan's homotopy formula, the authors proved  that,
if $D \subset \rn {N}$ is a convex set,
$1<p<\infty$, $1\le h \le N$, then  there exists a bounded linear map: 
\begin{equation*}
K_{\mathrm{Euc},h}  : L^p(D, {\bigwedge}\vphantom{!}^h)\to W^{1,p}_{\mathrm{Euc}}(D, {\bigwedge}\vphantom{!}^{h-1})
\end{equation*}
 that
is a homotopy operator, i.e.
\begin{equation}\label{may 4 eq:1}
	\omega =  dK_{\mathrm{Euc},h} \omega + K_{\mathrm{Euc},h+1}d\omega \qquad \mbox{for all 
	$\omega\in C^\infty (D, {\bigwedge}\vphantom{!}^h)$}.
\end{equation}
(see Proposition 4.1 and Lemma 4.2 in \cite{IL}). More precisely, $K_{\mathrm{Euc},h}$ has the form
\begin{equation}\label{10 maggio eq:1}
K_{\mathrm{Euc},h} \omega (x)= \int_D \psi(y)K_y\omega(x) \, dy,
\end{equation}
where $\psi \in \mc D(D)$, $\int_D\psi(y)\, dy=1$, and 
\begin{equation}\label{10 maggio eq:2}\begin{split}
&\Scal{K_y\omega(x)}{\xi_1\wedge\cdots
\wedge \xi_{h-1})}:=
\\&
\hphantom{xxxxxxxx}
\int_0^1t^{h-1}\Scal{\omega(y+t(x-y))}{(x-y)\wedge\xi_1\wedge\cdots
\wedge \xi_{h-1})}.
\end{split}\end{equation}
The definition \eqref{10 maggio eq:2} can be written as
$$
K_y \omega(x)=\int_{0}^{1}t^{\ell-1}\iota_{x-y}\omega(y_t)\,dt,
$$
where $y_t=y+t(x-y)$. Here, $\iota$ denotes the interior product of a differential form with a vector field,
i.e. 
$\iota: {\bigwedge}\vphantom{!}^{h+1} \to {\bigwedge}\vphantom{!}^h$ and is defined by
$$
\Scal{\iota_Y\omega}{v_1\wedge\cdots \wedge v_h} := \Scal{\omega}{Y\wedge v_1\wedge\cdots \wedge v_h}.
$$

Let us remind the following identity that follows straightforwardly from the relationship between the Lie derivative $\mc L_X$
along a vector field  $X$ of
a differential form 
and the interior product  of a vector field $Y$ and a differential form:
\begin{equation}\label{magic}
[\mc L_X, \iota_Y] = \iota_{[X,Y]}.
\end{equation}

\bigskip

The following theorem provides a  continuity result in $W^{k,p}$ of Iwaniec \& Lutoborski's kernel $K_{\mathrm{Euc}, \bullet} $,
though with a loss on domain. 

\begin{theorem}\label{chapeau}
Let $B=B(0,1)$ and $B'=B(0,2)$ be concentric Euclidean balls in $\rn N$. 
Then for $k\in \mathbb N$ and $p\in[1,\infty]$, Iwaniec-Lutoborski's homotopy $K_{\mathrm{Euc},h}$ is a bounded operator
\begin{equation*}
K_{\mathrm{Euc},h}  : W^{k,p}_{\mathrm{Euc}}(B', {\bigwedge}\vphantom{!}^\bullet )\to W^{k,p}_{\mathrm{Euc}}(B, {\bigwedge}\vphantom{!}^{\bullet-1})
\end{equation*}

\end{theorem}

\begin{proof}
For the sake of simplicity, from now on we omit the degree $h$ of the form and we write simply $ K_{\mathrm{Euc}}$.
We show that for every $k$-th order partial derivative $D^k$  there exist matrix valued kernels $M_1$ and $M_2$ on the ball of radius $2$ such that for every differential form $\omega$ on the unit ball, 
$$
D^k K_{\mathrm{Euc}}\omega=M_1\ast(D^k\omega)+M_2\ast(RD^{k-1}\omega),
$$
where $RD^{k-1}$ is a constant coefficient $(k-1)$-order differential operator and for all $h\in \R^n$, $|h|<2$, $i=1,2$,
$$
|M_i(h)|\leq C\,|h|^{1-N}.
$$

We set  $y_t=y+t(x-y)$. Iterating \eqref{magic}, we obtain
$$
D^k (\iota_{x-y}\omega(y_t))=t^k \iota_{x-y}D^k\omega(y_t)+t^{k-1}RD^{k-1}\omega(y_t),
$$
where $RD^{k-1}$ denotes the following $(k-1)$-order differential operator from $\ell$-forms to $\ell-1$ forms. 
If, for sake of simplicity, we take $D^k$ of the form
 $D^k=D_1\cdots D_k$,
$$
RD^{k-1}\omega=\sum_{i=1}^{k} \iota_{D_i}(D_1\cdots D_{i-1}D_{i+1}\cdots D_k\omega).
$$
Therefore
\begin{align*}
D^k K_{\mathrm{Euc}} \omega(x)
&=\int_{0}^{1}t^{\ell-1}\int_{B}\phi(y)D^k (\iota_{x-y}\omega(y_t))\,dy\,dt\\
&=\int_{B}\int_{0}^{1}t^{\ell-1}\phi(y)(t^k \iota_{x-y}D^k\omega(y_t)+t^{k-1}RD^{k-1}\omega(y_t))\,dy\,dt.
\end{align*}
Let us perform a change of variables $z=y_t$ and denote by $h=x-z$. T
hen $y=\frac{1}{1-t}z-\frac{t}{1-t}x=z-\frac{t}{1-t}h$, $x-y=(1-t)^{-1}h$, $dy=(1-t)^{-n}dz$, whence
\begin{align*}
D^k K_{\mathrm{Euc}}\omega(x)
&=\int_{B}\int_{0}^{1}t^{\ell-1}\phi(z-sh)(t^{k}(1-t)^{-1} \iota_{h}D^k\omega(z)\\
&+t^{k-1}RD^{k-1}\omega(z))\,(1-t)^{-N}\,dt\,dz.
\end{align*}
We treat both terms separately. The first one is
\begin{align*}
\int_{B}\int_{0}^{1}&t^{\ell-1}\phi(z-\frac{t}{1-t}h)t^{k}(1-t)^{-1} \iota_{h}D^k\omega(z)\,(1-t)^{-N}\,dt\,dz\\
&=\int_{B}\int_{0}^{1}t^{k+\ell-1}(1-t)^{-N-1}\phi(z-\frac{t}{1-t}h)\iota_{h}D^k\omega(z)\,dt\,dz\\
&=\int_{B}\langle \int_{0}^{\infty}(\frac{s}{1+s})^{k+\ell-1}(1+s)^{N-1}\phi(z-sh)\iota_{h}\,ds,D^k\omega(z)\rangle\,dz,
\end{align*}
where we have made the change of variables $s=\frac{t}{1-t}$. 
One recognizes the convolution of the $\Lambda^\ell$-valued function $D^k\omega$ with the matrix valued kernel
$$
M_1(z,h):=\int_{0}^{\infty}(\frac{s}{1+s})^{k+\ell-1}(1+s)^{N-1}\phi(z-sh)\iota_{h}\,ds.
$$
The second term is
\begin{align*}
\int_{B}\int_{0}^{1}&t^{\ell-1}\phi(z-\frac{t}{1-t}h)t^{k-1}RD^{k-1}\omega(z)\,(1-t)^{-N}\,dt\,dz\\
&=\int_{B}\int_{0}^{1}t^{k+\ell-2}(1-t)^{-N}\phi(z-\frac{t}{1-t}h)RD^{k-1}\omega(z)\,dt\,dz\\
&=\int_{B}\langle \int_{0}^{\infty}(\frac{s}{1+s})^{k+\ell-2}(1+s)^{N}\phi(z-sh)\,ds,RD^{k-1}\omega(z)\rangle\,dz.
\end{align*}
Again, this is the convolution of the $\Lambda^{\ell-1}$-valued function $SW^{k-1}\omega$ with the scalar kernel
$$
M_2(z,h):=\int_{0}^{\infty}(\frac{s}{1+s})^{k+\ell-2}(1+s)^{N-2}\phi(z-sh)\,ds.
$$
Since $\phi$ has compact support in $B$, in both cases, the integral stops no later that $2|h|^{-1}$, thus
\begin{align*}
|M_1(z,h)|&\leq C\,|h|\int_{0}^{2|h|^{-1}}(1+s)^{N-1}\,ds\leq C\,|h|^{1-N},\\
|M_2(z,h)|&\leq C\,\int_{0}^{2|h|^{-1}}(1+s)^{N-2}\,ds\leq C\,|h|^{1-N}.
\end{align*}
With Young's inequality, this implies that for all $p\in[1,\infty]$,
$$
\|D^k K_{\mathrm{Euc}}\omega\|_{L^p(B, {\bigwedge}\vphantom{!}^{\bullet-1})}
\leq C\,\left(\|\nabla^{k}\omega\|_{L^p(B', {\bigwedge}\vphantom{!}^{\bullet})}+\|\nabla^{k-1}\omega\|_{L^p(B', {\bigwedge}\vphantom{!}^{\bullet})}\right).
$$
Since this holds for every $k$-th order partial derivative,
$$
\|K_{\mathrm{Euc}}\omega\|_{W^{k,p}_{\mathrm{Euc}}(B, {\bigwedge}\vphantom{!}^{\bullet-1} )}
\leq 
C\,\|\omega\|_{W^{k,p}_{\mathrm{Euc}}(B', {\bigwedge}\vphantom{!}^\bullet )}.
$$
\end{proof}

Starting from \cite{IL}, in \cite{mitrea_mitrea_monniaux}, Theorem 4.1,  the authors define a compact homotopy operator $J_{\mathrm{Euc},h}$ in Lipschitz star-shaped  domains in Euclidean space $\rn {N}$, providing an explicit representation formula
for $J_{\mathrm{Euc},h}$, together with continuity properties among Sobolev spaces. More precisely:
\begin{theorem}\label{MMMTh}[(see \cite{mitrea_mitrea_monniaux}, formula (167))] if $D\subset \rn {N}$ is a star-shaped Lipschitz domain and {$1\le h\le N$}, then  there exists
$$
J_{\mathrm{Euc},h} : L^{p}(D, {\bigwedge}\vphantom{!}^h) \to W^{1,p}_{0, \mathrm{Euc}}(D, {\bigwedge}\vphantom{!}^{h-1})
$$
such that 
$$
\omega = dJ_{\mathrm{Euc},h}\omega + J_{\mathrm{Euc},h+1}d\omega \qquad \mbox{for all $\omega\in 
\mc D(D, {\bigwedge}\vphantom{!}^h)$}
$$
and for $1<p<\infty$ and $k\in \mathbb N\cup\{0\}$
$$
J_{\mathrm{Euc},h} : W^{k,p}_{0,\mathrm{Euc}}(D, {\bigwedge}\vphantom{!}^h) \to W^{k+1,p}_{0,\mathrm{Euc}}(D, {\bigwedge}\vphantom{!}^{h-1}).
$$

Furthermore, $J_{\mathrm{Euc},h}$ maps smooth compactly supported forms to smooth compactly supported forms.
\end{theorem}

 We need now construct a homotopy operator, fitting the intrinsic group structure, that can invert Rumin's differential $d_c$. To this aim take
 $D=B(e,1)=:B$ and $N=2n+1$. If $\omega\in C^\infty(B,E_0^h)$, then we set
\begin{eqnarray}\label{may 4 eq:2}
K=\Pi_{E_0}\circ \Pi_E \circ K_{\mathrm{Euc}}   \circ \Pi_E
\end{eqnarray}
(for the sake of simplicity, from now on we drop the index $h$ - the degree of the form -
writing, e.g., $K_{\mathrm{Euc}}$ instead of $K_{\mathrm{Euc},h}$).

Analogously, we can define
\begin{eqnarray}\label{may 31 eq:2}
J=\Pi_{E_0}\circ \Pi_E \circ J_{\mathrm{Euc}}   \circ \Pi_E.
\end{eqnarray}

Then $K$ and $J$ invert Rumin's differential $d_c$ on closed forms of the same degree. More
precisely, we have:

\begin{lemma}\label{homotopy 1} If $\omega$ is a smooth $d_c$-exact differential form, then
\begin{equation}\label{homotopy closed}
\omega = d_cK\omega \quad\mbox{if $1\le h\le 2n+1$}\qquad\mbox{and}\qquad   
\omega = d_cJ\omega \quad\mbox{if $1\le h\le 2n+1$.}
\end{equation}
In addition, if $\omega$ is compactly supported in $B$, then $J\omega$ is still compactly supported in $B$.
\end{lemma} 
For the proof of the lemma above we refer to Lemma 5.7 in \cite{BFP2}.

Imitating \cite{BFP2}, we are now able to prove
interior Poincar\'e inequality and Sobolev inequality for {Rumin} forms in the sense of
Definitions \ref{poincare def} and \ref{equiv Sobolev}.

\begin{theorem}\label{pq poincare} Take $\lambda >1$ and set  $B=B(e,1)$ and $B_\lambda=B(e,\lambda)$.
If $1\le h \le 2n+1$ then
\begin{itemize}
\item[i)] an interior $\he{}$-$\mathrm{Poincar\acute{e}}_{\infty ,\infty}(h)$ inequality holds with respect to the balls $B$ and $B_\lambda$;

\item[ii)] in addition, an interior
$\he{}$-$\mathrm{Sobolev}_{\infty,\infty}(h)$ inequality holds for $1\le h\le {2n+1}$.
\end{itemize}
\end{theorem}

\begin{proof} Consider the balls $B:=B(e,1) \Subset B(e,\lambda/2) \Subset B(e,\lambda)=:B_\lambda$, so that Proposition \ref{smoothing-A}
and Theorem \ref{chapeau}
can be applied to the couple $B(e,1), B(e,\lambda/2)$ and can be applied also to the couple $B(e,\lambda/2),  B(e,\lambda)$. Put $B_1:=B(e,\lambda/2)$.

\medskip

\noindent  i) Interior $\he{}$-$\mathrm{Poincar\acute{e}}_{\infty,\infty}(h)$ inequality: let $\omega \in L^\infty(B_\lambda,E_0^h)$ be $d_c$-closed.
By  \eqref{homotopy dec 27}, if we take therein $U:= B$ and $U':= B_\lambda$,  we can write
 \begin{equation}\label{9 nov}
\omega = d_cT\omega + S\omega \qquad\mbox{in $B$.}
\end{equation}
  By 
\eqref{dec 29 1} $S\omega\in 
\mc C^\infty(B,E_0^h)$ and $d_cS\omega =0$ since $d_c\omega = d_c^2T\omega + d_cS\omega$ in $B$
and $d_c\omega=0$ (by assumption).

 Thus we can apply \eqref{homotopy closed} to $S\omega$
and we get $S\omega = d_cKS\omega$, where $K$ is defined in \eqref{may 4 eq:2}.
In $B$, put now 
$$
\phi:= (KS+T)\omega.
$$ 
Trivially, 
\begin{equation}\label{dec 29 6}
d_c\phi = d_cKS\omega + d_cT\omega = S\omega +
d_cT\omega = \omega,
\end{equation}
 by \eqref{9 nov}.  

On the other hand,
\begin{equation}\label{dec 29 2}
\|\phi  \|_{L^\infty(B,E_0^{h-1})}  \le \|KS\omega\|_{L^\infty(B,E_0^{h-1})}+ \| T\omega\|_{L^\infty(B,E_0^{h-1})} .
\end{equation}
First of all, by \eqref{consiglio},
\begin{equation}\label{dec 29 3}
\| T\omega\|_{L^\infty(B,E_0^{h-1})} \le C \| \omega\|_{L^\infty(B,E_0^{h-1})}.
\end{equation}
Take now $q>2n+1$. By \cite{adams}, Theorem 4.12, keeping in mind that $\Pi_E$ is an operator of order 0 or 1,
depending on the degree of the form, we have:
\begin{equation}\label{dec 29 4}\begin{split}
\|KS & \omega\|_{L^\infty(B,E_0^{h-1})} \le C \|KS\omega\|_{W^{1,q}_{\mathrm{Euc}}(B,E_0^{h-1})} 
\\&
= C \|(\Pi_{E_0}\circ \Pi_E \circ K_{\mathrm{Euc}}   \circ \Pi_E) S\omega\|_{W^{1,q}_{\mathrm{Euc}}(B,E_0^{h-1})}
\\&
\le C \|(\Pi_E \circ K_{\mathrm{Euc}}   \circ \Pi_E) S\omega\|_{W^{1,q}_{\mathrm{Euc}}(B, {\bigwedge}\vphantom{!}^{h-1})}
\\&
\le C \|( K_{\mathrm{Euc}}   \circ \Pi_E) S\omega\|_{W^{2,q}_{\mathrm{Euc}}(B, {\bigwedge}\vphantom{!}^{h-1})}
\\&
\le C \| \Pi_E S\omega\|_{W^{2,q}_{\mathrm{Euc}}(B_1, {\bigwedge}\vphantom{!}^{h-1})} \qquad\mbox{(by Theorem \ref{chapeau})}
\\&
\le C \|S\omega\|_{W^{3,q}_{\mathrm{Euc}}(B_1,E_0^{h-1})} \le C  \|S\omega\|_{W^{3,\infty}_{\mathrm{Euc}}(B_1,E_0^{h-1})}
\\&
\le C  \|\omega\|_{L^{\infty}(B_\lambda,E_0^{h-1})} \qquad\mbox{(by \eqref{smoothing eq:1A}).}
\end{split}\end{equation}

Combining \eqref{dec 29 3} and \eqref{dec 29 4} it follows from \eqref{dec 29 2} that 
\begin{equation}\label{jan 2 1}
\|\phi  \|_{L^\infty(B,E_0^{h-1})}  \le C \|\omega\|_{L^{\infty}(B_\lambda,E_0^{h-1})},
\end{equation}
i.e. (keeping in mind \eqref{dec 29 6}), interior $\he{}$-$\mathrm{Poincar\acute{e}}_{\infty,\infty}(h)$ inequality holds.

\bigskip

ii) Interior $\he{}$-$\mathrm{Sobolev}_{\infty,\infty}(h)$ inequality: let $\omega \in L^\infty(B_\lambda,E_0^h)$ be $d_c$-closed
and compactly supported.
By  \eqref{homotopy dec 27}, if we take therein $U:= B_1$ and $U':= B_\lambda$,  we can write
 \begin{equation*}
\omega = d_cT\omega + S\omega \qquad\mbox{in $B_1$.}
\end{equation*}
  By 
\eqref{dec 29 1} $S\omega\in 
\mc C^\infty(B_1,E_0^h)$ and $d_cS\omega =0$ since $d_c\omega = d_c^2T\omega + d_cS\omega$ in $B_1$
and $d_c\omega=0$ (by assumption).

We notice now that $T\omega$ is supported in $B_1$ provided $R>0$ is small enough, so that, by \eqref{dec 29 1},
also $S\omega$ is supported in $B_1$. Thus, arguing as above, we can apply \eqref{homotopy closed} to $S\omega$
and we get $S\omega = d_c J S\omega$, where $J$ is defined in \eqref{may 31 eq:2}.
In $B_1$, put now 
$$
\phi:= (JS+T)\omega.
$$ 
We stress that, again by Lemma \ref{homotopy 1}, $JS\omega$ is compactly supported in $B_1$. Again as above, 
\begin{equation}\label{dec 29 7}
d_c\phi = d_cJS\omega + d_cT\omega = S\omega +
d_cT\omega = \omega.
\end{equation}
We can repeat now the arguments yielding the estimates \eqref{dec 29 4} and \eqref{jan 2 1}, 
replacing Theorem \ref{chapeau} by Theorem \ref{MMMTh}. Thus interior $\he{}$-$\mathrm{Sobolev}_{\infty,\infty}(h)$ inequality
follows.
\end{proof}

\section{Cohomology for annuli}\label{homotopy annuli}

The proof of $\he{}$-$\mathrm{Poincar\acute{e}}_{Q ,\infty}(h)$  (Theorem \ref{poincare infty}) given in Section \ref{caldo} relies basically on a duality argument and the dual inequality of \cite{BFP3} (see Theorem \ref{S introduction}). 

On the contrary, the proof of $\he{}$-$\mathrm{Sobolev}_{Q ,\infty}(h)$ (Theorem \ref{sobolev q infty}) requires a more sophisticated argument based on localization  on Kor\'anyi annuli. The present section is precisely devoted to prove that the $L^{\infty,\infty}$ cohomology 
of Rumin's closed forms vanishes on Kor\'anyi annuli. To this end, we prove first that de Rham $L^{\infty,\infty}$ cohomology of closed forms vanishes 
on Euclidean annuli. It follows that the same statement holds for suitable Kor\'anyi annuli (see Corollary \ref{koranyi annuli de rham}) end eventually the assertion is proven.

Let us start with the following definition.

\begin{definition}\label{aria condizionata}
Let $D\subset \mathbb R^{2n+1}$ be an open set. Let $s\in\mathbb{N}$
and $1\leq p,q\leq\infty$. If $1\le h\le 2n+1$, we define cohomology spaces
\begin{eqnarray*}
H^{s,q,p,h}_{\mathrm{de\ Rham}}(D)=(W^{s,p}_{\mathrm{Euc}}(D, \cov{h})\cap\mathrm{ker}d)/dW^{s,q}_{\mathrm{Euc}}(D, \cov{h-1}),
\end{eqnarray*}
and we denote by 
\begin{eqnarray*}
EH^{s,q,p,h}_{\mathrm{de\ Rham}}(D)=\mathrm{ker}(H^{s,q,p,h}_{\mathrm{de\ Rham}}(D)\to H^h(D))
\end{eqnarray*}
the cohomology of exact differential forms. Similar definitions hold with $d$ replaced with $d_c$, yielding the corresponding spaces 
{
$$H_{E_0}^{s,q,p,h}\quad \mbox{ and} \quad EH_{E_0}^{s,q,p,h}$$
}
for Rumin's differential forms.

If $s=0$ we shall write $H_{E_0}^{q,p,h} for H_{E_0}^{0,q,p,h}$.
\end{definition}

\begin{notation} If $0<s_1<s_2$ we denote by $A_{s_1,s_2}^\mathrm{Euc}$ the (Euclidean) annulus
$$
A_{s_1,s_2}^\mathrm{Euc} = B_\mathrm{Euc}(e,s_2)\setminus\overline{B_\mathrm{Euc}(e,s_1)}.
$$
Analogously, if $0<r_1<r_2$, we denote by $A_{r_1,r_2}$ the (Kor\'anyi) annulus
$$
A_{r_1,r_2} = B(e,r_2)\setminus\overline{B(e,r_1)}.
$$
\end{notation}

Given $0<r_1<r_2$, let $A_{r_1,r_2}$ be the (Kor\'anyi) annulus in $\he n$.
Put $\partial^+ A_{r_1.r_2}:= \partial B(e, r_2)$ and $\partial^- A_{r_1.r_2}:= \partial B(e, r_1)$. The meaning, in the Euclidean case, of 
$\partial^\pm A^\mathrm{Euc}_{r_1.r_2}$ is analogous.
Set $$V:= A_{1,2}$$ and $\sigma_1:=\frac 12 \min_{\partial^-V} |x| >0$ and $\sigma_2:=2 \max_{\partial^+V} |x|$. It turns out 
that $V\Subset A^\mathrm{Euc}_{\sigma_1,\sigma_2}=: \tilde V$.
Put now $\tilde V ':= A^\mathrm{Euc}_{\frac12 \sigma_1, 2\sigma_2}$, obviously $\tilde V\Subset \tilde V'$.
Finally set $\tau_1:=\frac 12 \min_{\partial^-\tilde V'} \rho(x) >0$ and we fix (once for all) $\tau_2 > \max_{\partial^+\tilde V'} \rho(x)$.
Then the Kor\'any annulus $ V ':=  A_{\frac12 \tau_1, 2\tau_2}$ satisfies
\begin{equation}\label{V V'}
V \Subset \tilde V \Subset \tilde V' \Subset V'.
\end{equation}

Notice that $\sigma_1= \frac 12 \min_{B(e,1)^c} |x|$, $\tau_1=\frac 12 \min_{B_{\mathrm{Euc}}(e,\frac12 \sigma_1)^c} \rho(x) $ and $\sigma_2= 2 \max_{B(e,2)} |x|$, $\tau_2=2 \max_{B_{\mathrm{Euc}}(e,2 \sigma_2)} \rho(x) $. 

\begin{definition}\label{admissible bis} With the notation introduced above, 
let $U=A_{s_1,s_2}$ and $U'=A_{r_1,r_2}$ be concentric Kor\'anyi annuli in $\he n$, $U\subset U'$.
We say that the couple $(U,U')$ is {\bf annulus-admissible} if,
with the notations of \eqref{V V'}, there exists $t>0$ such that 
$$
V' \subset \delta_t U' \qquad\mbox{and} \qquad \delta_t U \subset V.
$$
\end{definition}

 \begin{remark}\label{april 24}
A straightforward computation shows that the previous definition make sense.

Indeed, if $0<r_1<r_2< \frac{\tau_2}{\tau_1}r_1$ there exist $0<s_1<s_2$ such that the couple
$$
U:=A_{s_1,s_2}\qquad\mbox{and} \qquad U':=A_{r_1,r_2}
$$
is annulus admissible. More precisely the assertion holds provided 
$$
\frac{r_2}{\tau_2} < s_1 < \frac{r_1}{\tau_1}\qquad\mbox{or}\qquad
\frac{ 2r_2}{\tau_2} < s_2 < \frac{2r_1}{\tau_1}.
$$
\end{remark}

\begin{proposition}
\label{euclidean annuli}
Let $1\leq p\leq\infty$. Let $A^\mathrm{Euc}_{s_1,s_2} \Subset A^\mathrm{Euc}_{r_1,r_2}$ be concentric Euclidean annuli in $\R^{2n+1}$. Then the map 
$EH^{s,p,p,*}_{\mathrm{de\ Rham}}(A^\mathrm{Euc}_{r_1,r_2})\to EH^{s,p,p,*}_{\mathrm{de\ Rham}}(A^\mathrm{Euc}_{s_1,s_2})$ induced by the inclusion 
$A^\mathrm{Euc}_{s_1,s_2} \subset A^\mathrm{Euc}_{r_1,r_2}$ vanishes.
\end{proposition}

\begin{proof}
We use a diffeomorphism of $U'$ to $(-2,2)\times\Sigma$ mapping $U$ to $(-1,1)\times\Sigma$, where $\Sigma$ denotes the $2n$-sphere. Then we use Poincar\'e's homotopy formula in order to relate the cohomology of $(-2,2)\times\Sigma$ to the cohomology of $\Sigma$. Forms on the product can be written
\begin{eqnarray*}
\omega=a_t+dt\wedge b_t,
\end{eqnarray*}
where $a_t$ and $b_t$ are forms on $\Sigma$. Then 
\begin{eqnarray*}
d\omega=da_t+dt\wedge(\frac{\partial a_t}{\partial t}-db_t),
\end{eqnarray*}
where the right-hand side $d$ is the exterior differential on $\Sigma$. Assuming that $d\omega=0$, i.e. $da_t=0$ and $\frac{\partial a_t}{\partial t}=db_t$ for all $t\in(-2,2)$, set, for $\sigma\in\Sigma$ and $x\in(-1,1)$,
\begin{eqnarray*}
\gamma_x(t,\sigma)=\int_{x}^{t}b_u\,du.
\end{eqnarray*}
We observe that for all $p\geq 1$,
$$
\|\frac{1}{2}\int_{-1}^{1}\gamma_x\,dx\|_{W^{s,p}_{\mathrm{Euc}}}\leq \|\omega\|_{W^{s,p}_{\mathrm{Euc}}}
$$

By construction,
\begin{eqnarray*}
d\gamma_x=dt\wedge b_t+\int_{x}^{t}db_u\,du=\omega-a_x.
\end{eqnarray*}
Now assume that $\omega$ is exact, $\omega=d(e_t+dt\wedge f_t)$. Then $de_t=a_t$ for all $t\in(-2,2)$. Set
\begin{eqnarray*}
\gamma=\frac{1}{2}\int_{-1}^{1}(e_x+\gamma_x)\,dx,\quad\textrm{so that}\quad d\gamma=\omega.
\end{eqnarray*}
If $\omega\in W^{s,p}_{\mathrm{Euc}}$, so is each $\gamma_x$. On $\Sigma$, use the coexact primitive $e_x=\delta\Delta^{-1}a_x$
(see, e.g. \cite{GMS}, Section 2.5). Here $\Delta$ is the usual Hodge Laplacian on de Rham's differential $d$). Then, if $p<\infty$, 
\begin{eqnarray*}
\|\frac{1}{2}\int_{-1}^{1}e_x\,dx\|_{W^{s+1,p}_{\mathrm{Euc}}}\leq C\,\|\frac{1}{2}\int_{-1}^{1}a_x\,dx\|_{W^{s,p}_{\mathrm{Euc}}}\leq C'\,\|\omega\|_{W^{s,p}_{\mathrm{Euc}}}.
\end{eqnarray*}
If $p=\infty$, one picks $p>2n+1$, so that the Sobolev embedding theorem applies,
$$
\|\frac{1}{2}\int_{-1}^{1}e_x\,dx\|_{W^{s,\infty}_{\mathrm{Euc}}}\leq C\,\|\frac{1}{2}\int_{-1}^{1}e_x\,dx\|_{W^{s+1,p}_{\mathrm{Euc}}}.
$$
Obviously, $\|\omega\|_{W^{s,p}_{\mathrm{Euc}}}\leq C\,\|\omega\|_{W^{s,\infty}_{\mathrm{Euc}}}$. Hence the primitive $\gamma$ is bounded by $\omega$ in $W^{s,p}_{\mathrm{Euc}}$ norm in all cases.
This shows that the cohomology class of $\omega$ in $EH^{s,p,p,*}_{\mathrm{de\ Rham}}(U)$ vanishes.
\end{proof}

As a consequence of the previous result and keeping in mind Definition \ref{admissible bis}, we can prove the following corollary.

\begin{corollary}
\label{koranyi annuli de rham}
Let $U,U'$ be concentric Kor\'anyi annuli in $\he n$, $U\subset U'$
such that the couple $(U,U')$ is \emph{annulus-admissible}.
 Then the map 
$EH^{s,p,p,\bullet}_{\mathrm{de\ Rham}}(U')\to EH^{s,p,p,\bullet}_{\mathrm{de\ Rham}}(U)$ induced by the inclusion $U\subset U'$ vanishes
for $1\le p\le\infty$.

\end{corollary}

\begin{proof}

Suppose $U'\subset \delta_t V'$ and $\delta_t V\subset U$, where $V,V'$ and $t>0$ are as in Definition \ref{admissible bis}.
By \eqref{V V'} and Proposition \ref{euclidean annuli} we can conclude straightforwardly that
the map 
$EH^{s,p,p,*}_{\mathrm{de\ Rham}}(V')\to EH^{s,p,p,*}_{\mathrm{de\ Rham}}(V)$ induced by the inclusion $V\subset V'$ vanishes,
so that the map 
$EH^{s,p,p,*}_{\mathrm{de\ Rham}}(\delta_t U)\to EH^{s,p,p,*}_{\mathrm{de\ Rham}}(\delta_t U')$ vanishes. The assertion
follows by a pull-back argument.

\end{proof}

\begin{proposition}\label{p cohomology k-annuli}
 Let $U=A_{s_1,s_2}$ and $U''=A_{r_1,r_2}$ be concentric Kor\'any annuli in $\he{n}$. Assume
$(U,U'')$ are annulus-admissible (see Definition \ref{admissible bis}). Then the map 
$$EH_{E_0}^{\infty,\infty,\bullet}(A_{r_1,r_2})\to EH_{E_0}^{\infty,\infty,\bullet }(A_{s_1,s_2})$$ induced by inclusion $U\subset U''$ vanishes.
\end{proposition}

\begin{proof}
Let the annulus $U'$ be such that $U\subset U'\subset U''$ and such that the couple $(U,U')$ is still annulus-admissible
as in Definition \ref{admissible bis} (this is possible by Remark \ref{april 24}). 

Let  $\omega$ be a $d_c$-exact Rumin form on $U''$, 
which belongs to $L^\infty(U'', E_0^\bullet)$. 

Apply formula \eqref{smoothing eq:1A} of Proposition \ref{smoothing-A} with $s=5$. Then, if we set $S\omega=:\omega'\in W^{5,\infty}(U',E_0^{\bullet}) $, 
we have 
\begin{equation}\label{May 8}	
\omega=\omega'+d_c\alpha \quad\mbox{on $U'$}, \qquad\mbox{where} \qquad \alpha=T\omega\in L^\infty(U',E_0^{\bullet-1}).
\end{equation}

Consider $\omega'':=\Pi_E \omega'$. 
Obviously,  $\omega''=\Pi_E \omega''$. Moreover, by Theorem \ref{main rumin new}-iv),  $\Pi_{E_0}\Pi_E\Pi_{E_0} = \Pi_{E_0}$, and
$\Pi_{E_0}\omega'$ since
$\omega'$ is a Rumin form. Therefore
$$
\Pi_{E_0}\omega''= \Pi_{E_0}\Pi_E \omega'=\Pi_{E_0}\Pi_E\Pi_{E_0}\omega'=\omega'.
$$
Notice that $d\omega'' =0$ in $U'$. Indeed, since $\omega$ is $d_c$-exact,
then  $0= d_c\omega$ and hence $d_c\omega' =0$ in $U'$.
Therefore
$0= d_c\omega' = \Pi_{E_0}\Pi_E d\omega'$, so that
$0= \Pi_E\Pi_{E_0}\Pi_E d\omega' = \Pi_E d\omega' = d\Pi_e\omega'= d\omega''$
in $U'$ (keep in mind $d\Pi_E = \Pi_Ed$ by Theorem \ref{main rumin new}). 

In addition, $\omega''\in W^{4,\infty}(U',\Omega^\bullet)\subset W^{2,\infty}_{\mathrm{Euc}}(U', \Omega^\bullet)$.

According to Corollary \ref{koranyi annuli de rham}, there exists a differential form 
$\gamma\in W^{2,\infty}_{\mathrm{Euc}}(U)$ such that $\omega''=d\gamma$ on $U$. Hence
\begin{eqnarray*}
\omega''=\Pi_E \omega''=\Pi_Ed\gamma=d\Pi_E \gamma.
\end{eqnarray*}
If we set  $\eta=\Pi_{E_0}\Pi_E\gamma$, then in particular
$\eta\in L^\infty (U,E_0^{\bullet -1})$ and it follows that
\begin{equation*}\begin{split}
d_c\eta &= \Pi_{E_0}d\Pi_E \Pi_{E_0}\Pi_E\gamma = \Pi_{E_0}d \Pi_E\gamma 
\\&
=  \Pi_{E_0} \Pi_E d\gamma = \Pi_{E_0} \Pi_E\omega'' =   \Pi_{E_0} \omega'' = \omega'.
\end{split}\end{equation*} 
Hence, by \eqref{May 8},
$$
\omega=d_c(\eta+\alpha) \qquad\mbox{in $U$.}
$$
This shows that the cohomology class of the restriction of $\omega$ to $U$ vanishes in
 $EH^{\infty,\infty,\bullet}_{ E_0}(U)$.
\end{proof}

\begin{remark}\label{May 8 bis}
Repeating verbatim the proof of the previous theorem and keeping into account
\eqref{consiglio2} in Proposition \ref{smoothing-A}, when dealing with $(n+1)$-forms
the previous result guarantees the existence of a $W^{1,\infty}$-primitive, i.e.
the map 
$$
EH_{E_0}^{\infty,\infty,n+1}(A_{r_1,r_2})\to EH_{ E_0}^{1, \infty,\infty,n+1 }(A_{s_1,s_2})
$$ 
induced by the inclusion $U\subset U''$ vanishes.

\end{remark}

\section{Proof of Theorem \ref{sobolev q infty}}\label{proof}

We are now able to prove the Sobolev inequality as stated in Theorem \ref{sobolev q infty}.

The proof will be carried out starting from the corresponding Poincar\'e inequality by means of  localizations of our
estimates on a family of annuli via a suitable cut-off. Then a problem arises since the differential $d_c$ may have
order 1 or 2 according to the degree of the forms on which it acts.
Keeping in mind Remark \ref{commutators_remark}, for technical reasons, during the proof we are led to distinguish the case $h\neq n+2$ from the case $h=n+2$.

\begin{proof}[{\bf Proof of Theorem \ref{sobolev q infty}}]
We set
$$
q=q(h):=
\left\{
\begin{array}{lc}
\displaystyle Q
\quad &   \textrm{if}\,\,h\ne n;\\
\\
\displaystyle Q/2 &   \textrm{if}\,\,h=n.
\end{array}
\right. 
$$

Take $B:=B(e,1)$ and $B_\lambda:=B(e,\lambda)$ with  $1+\eps <\lambda< \frac{\tau_2}{\tau_1} (1+\eps)$, where $\eps >0$ and $\tau_1,\tau_2$ are the geometric 
constants introduced in Definition \ref{admissible bis}.

\medskip

 By Remark \ref{april 24} there exist $0<s_1<s_2$ such that for two concentric annuli $A_{1+\eps,\lambda}$ and $A_{s_1,s_2}$,
 the
couple $(A_{1+\eps,\lambda}\,,\, A_{s_1,s_2})$ is annulus-admissible in the sense of Definition \ref{admissible bis}
and
$$
A_{s_1,s_2}\subset A_{1+\eps,\lambda}\subset B_\lambda\setminus B.
$$
If we set
$U':=A_{1+\eps,\lambda}$ and $U:=A_{s_1,s_2}$,
the inclusion above reads as
$$
U\subset U'\subset B_\lambda\setminus B.
$$

Let $\alpha\in L^{q}(B,E_0^{h})$ be a  compactly supported $d_c$-exact $h$-form on the unit ball $B$. If $h=2n+1$, this implies that $\int_{B}\alpha=0$. Otherwise, this simply means that $d_c\alpha=0$ in $B$. We continue $\alpha$ by zero on $\he n\setminus B$. 

 We apply $\he{}$-$\mathrm{Poincar\acute e}_{q,\infty}(h)$ in $2B_\lambda$ (see Theorem \ref{poincare infty}),
and we find $\gamma\in L^\infty(B_\lambda, E_0^{h-1})$ such that 
\begin{equation}\label{8 maggio 2}
d_c\gamma=\alpha\quad \mbox{in $B_\lambda$}\qquad\mbox{and}\qquad \|\gamma\|_{L^\infty(B_\lambda,E_0^{h-1})}\leq C\,\|\alpha\|_{L^q(2B_\lambda, E_0^{h})} = \|\alpha\|_{L^q(B, E_0^{h})}.
\end{equation} 
We emphasize here that the exponent $q$ in \eqref{8 maggio 2} equals   $Q$
{if} $h\ne n$
and  $Q/2$ {if} $ h=n$.

As announced above, we have to distinguish two cases: $h\neq n+2$ and $h=n+2$.
Since in $U'\subset B_\lambda\setminus B$ we have $d_c \gamma = 0$ in $U'$. Furthermore, if $h=2n+1$, 
$$
\int_{\partial B}\gamma=\int_{B}\alpha=0,
$$
which implies that $\gamma$ is exact on $B_\lambda\setminus B$.
Hence
by Proposition \ref{p cohomology k-annuli}, if $h-2\neq n$,
there exists a $(h-2)$-form $\gamma'$ on $U$ such that 
\begin{equation}\label{8 maggio 1}
d_c\gamma'=\gamma\qquad\mbox{in $U$ and}\qquad \|\gamma'\|_{L^\infty(U,E_0^{h-2})} \leq C\,\|\gamma\|_{L^\infty(U',E_0^{h-1})} .
\end{equation}
On the other hand, if $h-2=n$ then, by Remark \ref{May 8 bis},
there exists
 a $\gamma' \in  W^{1,\infty} (U, E_0^{n})  $ such that $d_c \gamma'=\gamma$ in $U$ and

\begin{equation}\label{uffa}
\|\gamma'\|_{W^{1,\infty} (U, E_0^{n})} \leq C\,\|\gamma\|_{L^\infty(U', E_0^{n+1})}.
\end{equation}

Let $\zeta$ be a smooth function in $B_\lambda$ vanishing in 
 $c_1 B$ with $s_1 < c_1 < s_2$, 
such that $\zeta\equiv 1$ outside of $c_2B$,
where $c_1<c_2 <s_2$.
We stress that $\gamma'$ is defined on $U$, and $\zeta\gamma'$
is  supported outside of a  neighborhood  of $\overline{c_1B}$ and therefore can be continued by
0 on all the ball $c_1B$ and then is defined on all of $s_2 B$.

We set 
\begin{equation}\label{11gennaio}\beta:=\gamma-d_c(\zeta\gamma') \end{equation}
(that is still defined on all of $s_2B$).
Now on $ s_2 B\setminus \overline{c_2B}  = U\setminus  \overline{c_2B}$
we have
$$
\beta = \gamma - d_c\gamma' \equiv 0,
$$
so that $\beta$ is compactly supported in $s_2 B$ and can be continued by
0 to a compactly supported form in $B_\lambda$. 

In addition, by \eqref{8 maggio 2},
$$
d_c\beta = d_c\gamma = \alpha\qquad \mbox{in $B_\lambda$.}
$$

\begin{center}
\includegraphics[width=50mm, height=100mm]{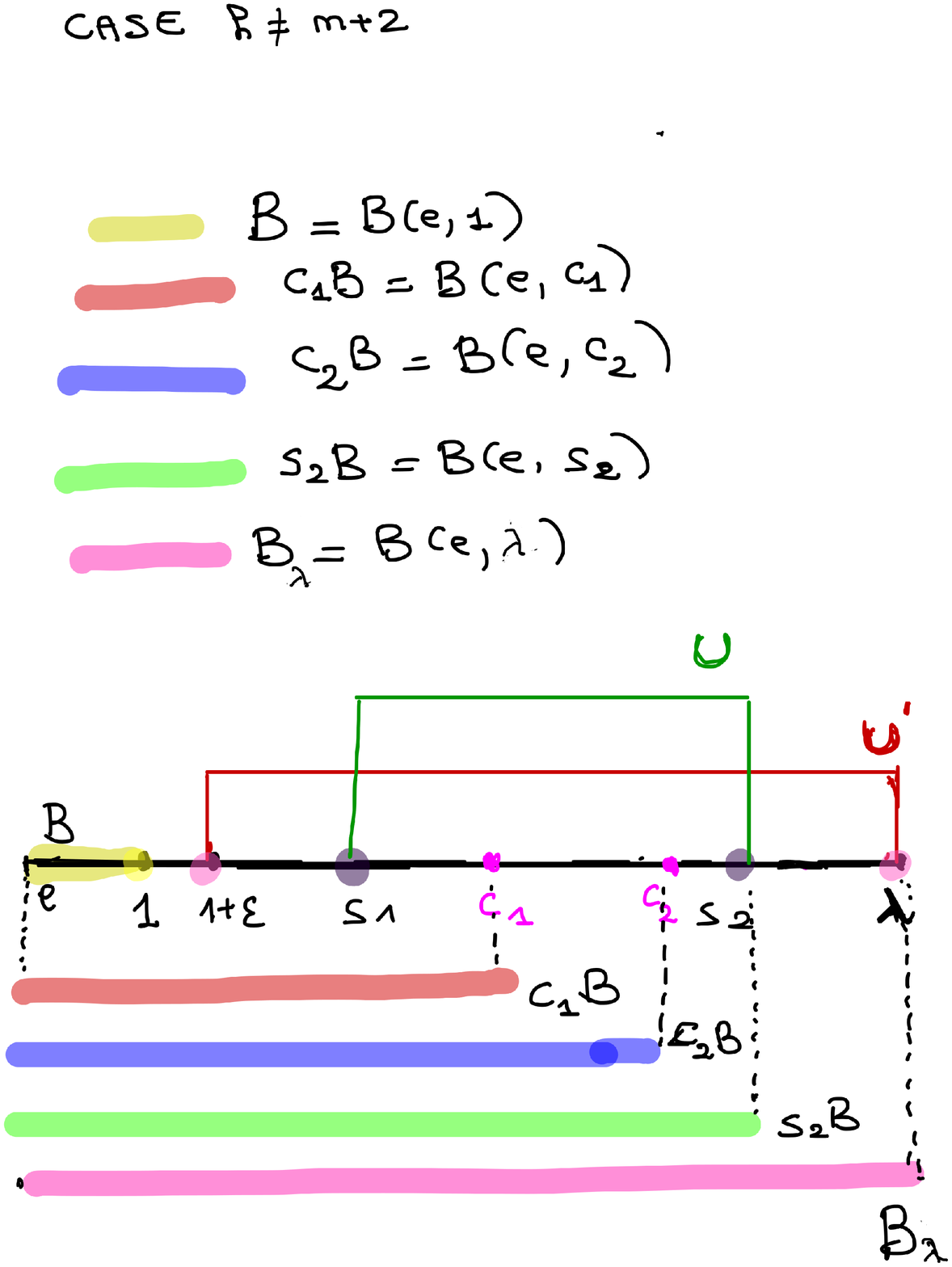}

\end{center}

By Remark \ref{commutators_remark}, keeping into account that 
$\zeta\gamma'\in E_0^{h-2}$, if $h-2\neq n$, we have
\begin{equation*}\begin{split}
\| \beta \|_{L^\infty(B_\lambda,E_0^{h-1})}  & =  \| \beta \|_{L^\infty(s_2B,E_0^{h-1}))}    \le \| \gamma \|_{L^\infty(B_\lambda,E_0^{h-1}))} + \| d_c( \zeta\gamma') \|_{L^\infty(s_2B,E_0^{h-1})}
\\&
= \| \gamma \|_{L^\infty(B_\lambda,E_0^{h-1})} + \| d_c( \zeta\gamma') \|_{L^\infty(U,E_0^{h-1})}
\\&
 \le \| \gamma \|_{L^\infty(B_\lambda,E_0^{h-1})} +  \|  d_c\gamma' \|_{L^\infty(U,E_0^{h-1})}
+  C\|\gamma' \|_{L^\infty(U, E_0^{h-2})}
\\&
\le \| \gamma \|_{L^\infty(B_\lambda,E_0^{h-1})} +\| \gamma \|_{L^\infty(U,E_0^{h-1})}+ C \| \gamma \|_{L^\infty(U',E_0^{h-1})}   \qquad\mbox{(by \eqref{8 maggio 1})}
\\&
\le C \| \gamma \|_{L^\infty(B_\lambda,E_0^{h-1})} \le C\|\alpha\|_{L^q(B,E_0^{h})}  \qquad\mbox{(by \eqref{8 maggio 2}).}
\end{split}\end{equation*}
Thus,  $\he{}$-$\mathrm{Sobolev}_{q,\infty}(h)$ holds for $h\neq n+2$.

On the other hand, when $h-2=n$ , keeping into account
Remark \ref{commutators_remark}-ii),  we have
\begin{equation*}\begin{split}
\| \beta \|_{L^\infty(B_\lambda,E_0^{h-1})}  & =  \| \beta \|_{L^\infty(s_2B,E_0^{h-1}))}    \le \| \gamma \|_{L^\infty(B_\lambda,E_0^{h-1}))} + \| d_c( \zeta\gamma') \|_{L^\infty(s_2B,E_0^{h-1})}
\\&
= \| \gamma \|_{L^\infty(B_\lambda,E_0^{h-1})} + \| d_c( \zeta\gamma') \|_{L^\infty(U,E_0^{h-1})}
\\&
 \le \| \gamma \|_{L^\infty(B_\lambda,E_0^{h-1})} +  \|  d_c\gamma' \|_{L^\infty(U,E_0^{h-1})}
+  C\|\gamma' \|_{W^{1,\infty}(U, E_0^{h-2})}
\\&
\le \| \gamma \|_{L^\infty(B_\lambda,E_0^{h-1})} +\| \gamma \|_{L^\infty(U,E_0^{h-1})}+ C \| \gamma \|_{L^\infty(U',E_0^{h-1})}   \qquad\mbox{(by \eqref{uffa})}
\\&
\le C \| \gamma \|_{L^\infty(B_\lambda,E_0^{h-1})} \le C\|\alpha\|_{L^q(B,E_0^{h})}  \qquad\mbox{(by \eqref{8 maggio 2}).}
\end{split}\end{equation*}
Thus, by Definition
\ref{equiv Sobolev},  $\he{}$-$\mathrm{Sobolev}_{q,\infty }(n+2)$ holds.
\end{proof}

\section{Appendix: Rumin's complex}\label{appendix}
Coherently with the notations introduced through the paper,
we set (see  \eqref{campi W})
\begin{equation*}
\omega_i:=dx_i, \quad \omega_{i+n}:= dy_i \quad { \mathrm{and} }\quad \omega_{2n+1}:= \theta, \quad \text
{for }i =1, \dots, n
\end{equation*}
we denote by $\scalp{\cdot}{\cdot}{} $ the
inner product in $\covH 1$  that makes $(dx_1,\dots, dy_{n},\theta  )$ 
an orthonormal basis.

{ 
We put
$       \vetH 0 := \covH 0 =\R $
and, for $1\leq h \leq 2n+1$,
\begin{equation*}
\begin{split}
         \covH h& :=\mathrm {span}\{ \omega_{i_1}\wedge\dots \wedge \omega_{i_h}:
1\leq i_1< \dots< i_h\leq 2n+1\}
.
\end{split}
\end{equation*}
In the sequel we shall denote by $\Theta^h$ the basis of $ \covH h$ defined by
$$
\Theta^h:= \{ \omega_{i_1}\wedge\dots \wedge \omega_{i_h}:
1\leq i_1< \dots< i_h\leq 2n+1\}.
$$
To avoid cumbersome notations, if $I:=({i_1},\dots,{i_h})$, we write
$$
\omega_I := \omega_{i_1}\wedge\dots \wedge \omega_{i_h}.
$$
The  {{ inner}} product $\scal{\cdot}{\cdot}$ on $ \covH 1$ yields naturally a {{ inner}} product 
$\scal{\cdot}{\cdot}$ on $ \covH h$
making $\Theta^h$ an orthonormal basis.

The volume $(2n+1)$-form $ \theta_1\wedge\cdots\wedge \theta_{ 2n+1}$
 will be also
written as $dV$.

Throughout this paper, the elements of $\cov h$ are identified with \emph{left invariant} differential forms
of degree $h$ on $\he n$. 

\begin{definition}\label{left} A $h$-form $\alpha$ on $\he n$ is said left invariant if 
$$\tau_q^\#\alpha
=\alpha\qquad\mbox{for any $q\in\he n$.}
$$
Here $\tau_q^\#\alpha$ denotes the pull-back of $\alpha$ through the left translation $\tau_q$.
\end{definition}

The same construction can be performed starting from the vector
subspace $\mfrak h_1\subset \mfrak h$,
obtaining the {\it horizontal $h$-covectors} 
\begin{equation*}
\begin{split}
         \covh h& :=\mathrm {span}\{ \omega_{i_1}\wedge\dots \wedge \omega_{i_h}:
1\leq i_1< \dots< i_h\leq 2n\}.
\end{split}
\end{equation*}
It is easy to see that 
$$
\Theta^h_0 := \Theta^h \cap  \covh h
$$ 
provides an orthonormal
basis of $ \covh h$.

Keeping in mind that the Lie algebra $\mathfrak h$ can be identified with the
tangent space to $\he n$ at $x=e$ (see, e.g. \cite{GHL}, Proposition 1.72), 
starting from $\cov h$ we can define by left translation  a fiber bundle
over $\he n$  that we can still denote by $\cov h$. We can think of $h$-forms as sections of 
$\cov h$. We denote by $\Omega^h$ the
vector space of all smooth $h$-forms.

We already pointed out in Section \ref{preliminary} that the stratification
of the Lie algebra $\mfrak h$ yields a lack of homogeneity of de Rham's exterior differential
with respect to group dilations $\delta_\lambda$.  Thus, to keep into account the different degrees
of homogeneity of the covectors when they vanish on different layers of the
stratification, we introduce the notion of {\sl weight} of a covector as follows.
}

\begin{definition}\label{weight} If $\eta\neq 0$, $\eta\in \covh 1$,  
 we say that $\eta$ has \emph{weight $1$}, and we write
$w(\eta)=1$. If $\eta = \theta$, we say $w(\eta)= 2$.
More generally, if
$\eta\in \covH h$, {  $\eta\neq 0$, }we say that $\eta$ has \emph {pure weight} $p$ if $\eta$ is
a linear combination of covectors $\omega_{i_1}\wedge\cdots\wedge\omega_{i_h}$
with $w(\omega_{i_1})+\cdots + w(\omega_{ i_h})=p$.
\end{definition}

Notice that, if $\eta,\zeta \in \covH h$ and $w(\eta)\neq w(\zeta)$, then
$\scal{\eta}{\zeta}=0$ (see \cite{BFTT}, Remark 2.4).  We notice also that
$w(d\theta) = w(\theta)$.

We stress that generic covectors may fail to have a pure weight: it is enough to
consider $\he 1$ and the covector $dx_1+\theta\in \covH{1}$. However, the
following result holds
(see \cite{BFTT}, formula (16)):
\begin{equation}\label{dec weights}
\covH h = \covw {h}{h}\oplus \covw {h}{h+1} =  \covh h\oplus \Big(\covh {h-1}\Big)\wedge \theta,
\end{equation}
where $\covw {h}{p}$ denotes the linear span of the $h$-covectors of weight $p$.
By our previous remark, the decomposition \eqref{dec weights} is orthogonal.
In addition, since the elements of the basis $\Theta^h$ have pure weights, a basis of
$ \covw {h}{p}$ is given by $\Theta^{h,p}:=\Theta^h\cap \covw {h}{p}$
(such a basis is usually called an adapted basis). 

As above, starting from  $\covw {h}{p}$, we can define by left translation  a fiber bundle
over $\he n$  that we can still denote by $\covw {h}{p}$. 
Thus, if we denote by  $\Omega^{h,p} $ the vector space of all
smooth $h$--forms in $\he n$ of  weight $p$, i.e. the space of all
smooth sections of $\covw {h}{p}$, we have
\begin{equation}\label{deco forms}
\Omega^h = \Omega^{h,h}\oplus\Omega^{h,h+1} .
\end{equation}

\noindent{\bf{Definition of Rumin's complex}}

Let us give a short introduction to Rumin's complex. For a more detailed presentation we
refer to Rumin's papers \cite{rumin_grenoble} following verbatim the presentation of \cite{BFP2}. Here we follow the presentation of \cite{BFTT}.

The exterior differential $d$ does not preserve weights. It splits into
\begin{eqnarray*}
d=d_0+d_1+d_2
\end{eqnarray*}
where $d_0$ preserves weight, $d_1$ increases weight by 1 unit and $d_2$ increases weight by 2 units.

More explicitly,
let $\alpha\in \Omega^{h}$ be a (say) smooth form
of pure weight $h$. We can write
$$
\alpha= \sum_{\omega_I\in\Theta^{h}_0}\alpha_{I}\, \omega_I 
,\quad
\mbox{with } \alpha_I \in \mc C^\infty (\he n).
$$
Then
$$
d\alpha= \sum_{\omega_I\in\Theta^{h}_0}\sum_{j=1}^{2n} (W_j\alpha_{I})\, \omega_j\wedge\omega_I + 
\sum_{\omega_I\in\Theta^{h}_0} (T \alpha_{I})\, \theta\wedge \omega_I = d_1\alpha +  d_2\alpha,
$$
and $d_0\alpha =0$. On the other hand, if $\alpha\in \Omega^{h,h+1}$ has pure weight $h+1$, then 
$$
\alpha =  \sum_{\omega_J\in\Theta^{h-1}_0}\alpha_{J}\, \theta\wedge\omega_J,
$$
and
$$
d\alpha= \sum_{\omega_J\in\Theta^{h}_0}\alpha_J\,d\theta\wedge\omega_J + \sum_{\omega_J\in\Theta^{h}_0}\sum_{j=1}^{2n} (W_j\alpha_{J})\, \omega_j\wedge\theta\wedge\omega_I 
=d_0\alpha+d_1\alpha,
$$
and $d_2\alpha=0$.

It is crucial to notice that  $d_0$ is an algebraic operator, in the sense that
for any real-valued $f\in\mc C^\infty (\he n)$ we have
$$
d_0(f\alpha)= f d_0\alpha,
$$
so that its action can be identified at any point with the action of a linear
operator from   $\cov h$ to $\cov {h+1}$ (that we denote again by $d_0$). 

Following M. Rumin (\cite{rumin_grenoble}, \cite{rumin_cras}) we give the following definition:
\begin{definition}\label{E0}
If $0\le h\le 2n+1$, keeping in mind that $\cov h$ is endowed with a canonical
{{ inner}} product, we set
$$
E_0^h:= \ker d_0\cap (\mathrm{Im}\; d_0)^{\perp}.
$$
Straightforwardly, $E_0^h$ inherits from $\cov h$ the
{{ inner}} product. 
\end{definition}

As above,  $E_0^\bullet$ defines by left translation a fibre bundle over $\he n$,
that we still denote by $E_0^\bullet$. To avoid cumbersome notations,
we denote also by  $E_0^\bullet$ the space of sections of this fibre bundle.

Let $L: \cov h \to \cov{h+2}$ the Lefschetz operator defined by
\begin{equation}\label{lefs}
L\, \xi = d\theta\wedge\xi.
\end{equation}
Then the spaces $E_0^\bullet$  can be defined explicitly as follows:

\begin{theorem}[see \cite{rumin_jdg}, \cite{rumin_gafa}] \label{rumin in H} We have:
\begin{itemize}
\item[i)] $E_0^1= \covh{1}$;
\item[ii)]  if $2\le h\le n$, then $E_0^h= \covh{h}\cap \big(\covh{h-2}\wedge d\theta\big)^\perp$
 (i.e. $E_0^h$ is the space of the so-called \emph{primitive covectors} of $\covh h$);
\item[iii)]  if $n< h\le 2n+1$, then $E_0^h = \{\alpha = \beta\wedge\theta, \; \beta\in \covh{h-1},
\; \gamma\wedge d\theta =0\} = \theta\wedge\ker L$;
\item[iv)]  if $1<h\le n$, then $N_h:=\dim E_0^h = \binom{2n}{h} - \binom{2n}{h-2}$;
\item[v)]  if $\ast$ denotes the Hodge duality associated with the {{ inner}} product in $\cov{\bullet}$
and the volume form $dV$, then $\ast E_0^h = E_0^{2n+1-h}$.
\end{itemize} 
Notice that all forms in $E_0^h$ have weight $h$ if $1\le h\le n$ and
weight $h+1$ if $n< h\le 2n+1$.

\end{theorem}

A further geometric interpretation (in terms of decomposition of $\mathfrak h$ and of graphs
within $\he n$) can be found in \cite{FS2}.

Notice that there exists a left invariant orthonormal basis 
\begin{equation}\label{basis E0}
\Xi_0^h=\{\xi_1^h,\dots, \xi_{\mathrm{dim}\, E_0^h}^h\}
\end{equation}
of $E_0^h$ that is adapted to the filtration
\eqref{dec weights}. Such a basis is explicitly constructed by {  induction} in \cite{BBF}.

The core of Rumin's theory consists in the construction of a suitable ``exterior differential''
$d_c: E_0^h \to E_0^{h+1}$ making $\mc E_0:= (E_0^\bullet, d_c)$ a complex homotopic
to the de Rham complex.

Let us sketch Rumin's construction: first the next result (see \cite{BFTT}, Lemma 2.11 for a proof) allows us to define a (pseudo) inverse of $d_0$ : 
\begin{lemma}\label{d_0} If $1\le h\le n$, then $\ker d_0 = \covh{h}$.
Moreover, if $\beta\in \covH {h+1}$, then there exists a unique $\gamma\in
\covH{h}\cap (\ker d_0)^\perp$ such that
$$
d_0\gamma-\beta\in \mc R(d_0)^\perp.
$$

\end{lemma}
With the notations of the previous lemma, we set $$\gamma :=d_0^{-1}\beta.$$
We notice that $d_0^{-1}$ preserves the weights.

The following  theorem summarizes the construction of 
the intrinsic differential $d_c$ (for details, see \cite{rumin_grenoble}
and \cite{BFTT}, Section 2) .
\begin{theorem}\label{main rumin new}
The de Rham complex $(\Omega^\bullet,d)$ 
splits into the direct sum of two sub-complexes $(E^\bullet,d)$ and
$(F^\bullet,d)$, with
$$
E:=\ker d_0^{-1}\cap\ker (d_0^{-1}d)\quad\mbox{and}\quad
F:= \mc R(d_0^{-1})+\mc R (dd_0^{-1}).
$$
Let $\Pi_E$ be the projection on $E$ along $F$ (that
is not an orthogonal projection). We have
\begin{itemize}
\item[i)]   If $\gamma\in E_0^{h}$,  then
\begin{itemize}
\item[$\bullet$] $
\Pi_E\gamma=\gamma -d_0^{-1}
d_1 \gamma$ if $1\le h\le n$;
\item[$\bullet$] $
\Pi_E\gamma=\gamma $ if $h>n$.
\end{itemize}
\item[ii)] $\Pi_{E}$ is a chain map, i.e.
$$
d\Pi_{E} = \Pi_{E}d.
$$
\item[iii)] Let $\;\Pi_{E_0}$ be the orthogonal projection from $\covH{*}$
on $E_0^\bullet$, then
\begin{equation}\label{PiE0 project}
\Pi_{E_0} = Id - d_0^{-1}d_0-d_0d_0^{-1}, \quad
\Pi_{E_0^\perp} =  d_0^{-1}d_0 + d_0d_0^{-1}.
\end{equation}
\item[iv)] $\Pi_{E_0}\Pi_{E}\Pi_{E_0}=\Pi_{E_0}$ and
$\Pi_{E}\Pi_{E_0}\Pi_{E}=\Pi_{E}$.

\end{itemize}

\noindent Set now
 $$d_c=\Pi_{E_0}\, d\,\Pi_{E}: E_0^h\to E_0^{h+1}, \quad h=0,\dots,2n.$$
 We have:
\begin{itemize}
\item[v)] $d_c^2=0$;
\item[vi)] the complex $E_0:=(E_0^\bullet,d_c)$ is homotopic to the de Rham complex;
\item[vii)] $d_c: E_0^h\to E_0^{h+1}$ is a homogeneous differential operator in the 
horizontal derivatives
of order 1 if $h\neq n$, whereas $d_c: E_0^n\to E_0^{n+1}$ is an homogeneous differential operator in the 
horizontal derivatives
of order 2;
\item[viii)] on forms of degree $h>n$ we have $d_c= d$. Indeed, if $\gamma\in E_0^h$ with $h>n$, then,
by i) and iv)  
$$
d_c\gamma = \Pi_{E_0}\Pi_E d\gamma  = \Pi_E\Pi_{E_0}\Pi_E d\gamma
= \Pi_E d\gamma = d \Pi_E \gamma = d\gamma,
$$
(see also \cite{BFT2}).
\item[ix)] on forms of degree $h=n$, $\Pi_E-Id_{E_0^n} =-d_0^{-1}
d_1$ raises weight by one unit, i.e. it maps $E_0^n\subset\bigwedge^{n,n}$ to $\bigwedge^{n,n+1}$.
\end{itemize}
\end{theorem}

The next remarkable property of Rumin's complex is its invariance under contact transformations. 
In particular,

\begin{proposition}\label{pull} If we write a form $\alpha = \sum_j \alpha_j \xi_j^h$ in coordinates with respect to a left-invariant basis of $E_0^h$
(see \eqref{basis E0}) we have:
\begin{equation}\label{pull trasl}
\tau_q^\# \alpha = \sum_j ( \alpha_j \circ \tau_q)\xi_j^h
\end{equation}
for all $q\in \he n$.
In addition, for $t>0$,
\begin{equation}\label{pull dil 1}
\delta_t^\# \alpha =  t^h \sum_j ( \alpha_j \circ \delta_t)\xi_j^h\qquad\mbox{if $1\le h \le n$}
\end{equation}
and 
\begin{equation}\label{pull dil 2}
\delta_t^\# \alpha =  t^{h+1} \sum_j ( \alpha_j \circ\delta_t)\xi_j^h\qquad\mbox{if $n+1\le h \le 2n+1$}\,.
\end{equation}
\end{proposition}

\section*{Acknowledgments}

A.B. and B.F.  are supported by the University of Bologna, funds for selected research topics, and by MAnET Marie Curie
Initial Training Network, and by 
GNAMPA of INdAM (Istituto Nazionale di Alta Matematica ``F. Severi''), Italy.

P.P. is supported by MAnET Marie Curie
Initial Training Network, by Agence Nationale de la Recherche, ANR-10-BLAN 116-01 GGAA and ANR-15-CE40-0018 SRGI. P.P. gratefully acknowledges the hospitality of Isaac Newton Institute, of EPSRC under grant EP/K032208/1, and of Simons Foundation.

\bibliographystyle{amsplain}

\bibliography{BFP5_finale}

\bigskip
\tiny{
\noindent
Annalisa Baldi and Bruno Franchi 
\par\noindent
Universit\`a di Bologna, Dipartimento
di Matematica\par\noindent Piazza di
Porta S.~Donato 5, 40126 Bologna, Italy.
\par\noindent
e-mail:
annalisa.baldi2@unibo.it, 
bruno.franchi@unibo.it.
}

\medskip

\tiny{
\noindent
Pierre Pansu 
\par\noindent  Universit\'e Paris-Saclay, CNRS, Laboratoire de math\'ematiques d'Orsay
\par\noindent  91405, Orsay, France.
\par\noindent 
e-mail: pierre.pansu@universite-paris-saclay.fr
}

\end{document}